\documentclass[12pt,twoside]{article}
\usepackage[all]{}
\usepackage {amssymb,latexsym,amsthm,amscd,amsmath}
\newtheorem{theorem}{Theorem}[section]
\newtheorem{lemma}{Lemma}[section]
\newtheorem{corollary}{Corollary}[section]
\newtheorem{proposition}{Proposition}[section]
\newtheorem*{definition}{Definition}
\newtheorem*{theorem*}{Theorem}
\theoremstyle{definition}

\numberwithin{equation}{subsection}

\newcommand{\ignore}[1]{}

\newcommand{\mynote}[1]{}
\begin{document}
\setcounter{section}{0}
% document information
\title{\bf Automorphisms of Albert algebras and a conjecture of Tits and Weiss}
%\author{Maneesh Thakur}
\author{Maneesh Thakur \\ \small {Indian Statistical Institute, 7-S.J.S. Sansanwal Marg} \\ \small {New Delhi 110016, India} \\ \small {e-mail: maneesh.thakur@
gmail.com}}
%\keywords{Automorphisms, Albert algebras, Structure group, inner structure group, Kneser-Tits}
\date{}
\maketitle
%\subjclass{2010 Math. Subject Classification. Primary 20G15, Secondary 17C30}
\begin{abstract}
\noindent
\it{ Let $k$ be an arbitrary field. The main aim of this paper is to prove the Tits-Weiss conjecture for Albert division algebras over $k$ which are pure first Tits constructions. This conjecture asserts that for an Albert division algebra $A$ over a field $k$, every norm similarity of $A$ is inner modulo scalar multiplications. It is known that $k$-forms of $E_8$ with index $E^{78}_{8,2}$ and anisotropic kernel a 
strict inner $k$-form of $E_6$ correspond bijectively (via Moufang hexagons) to Albert division algebras over $k$. 
The Kneser-Tits problem for a form of $E_8$ as above is equivalent to the Tits-Weiss conjecture (see \cite{TW}). Hence we provide a solution to the Kneser-Tits problem for forms of $E_8$ arising from pure first Tits construction Albert division algebras. As an application, we prove that 
for $G={\bf Aut}(A),~G(k)/R=1$, where $A$ is a pure first construction Albert division algebra over $k$ and $R$ stands for $R$-equivalence in the sense of Manin (\cite{M}).}
\end{abstract}   
%%%%%%%%%%%%%%%%%%%%%%%%%%%%%%%%%%%%%%%%%%%
\section{\bf Introduction}
The main result of this paper is a solution to the Tits-Weiss conjecture for Albert division algebras over $k$, which are pure first Tits constructions. As a consequence, we solve the Kneser-Tits problem for $k$-forms of $E^{78}_{8,2}$, which have anisotropic kernel a strict inner $k$-form of $E_6$ associated to Albert division algebras which are pure first Tits constructions.  
We study some group theoretic properties of the structure group and the group of automorphisms of an Albert algebra over a given field, which we proceed to describe below briefly. 

Let $A$ be an Albert division algebra over $k$ and $G={\bf Aut}(A)$ denote the algebraic group of automorphisms of $A$. Then $G$ is defined and anisotropic over $k$, of type $F_4$. We classify connected reductive algebraic subgroups of $G$ (see Section 6). Let $L\subset A$ be a cubic subfield and $G_L$ be the algebraic subgroup of $G$ consisting of automorphisms of $A$ fixing $L$ pointwise. It is known that $G_L$ is an algebraic group defined over $k$, of type $D_4$ 
(see chapters (VIII), (IX) of \cite{KMRT}). We prove that $G_L$ is generated, as an algebraic group, by 
$G_{D_1}$ and $G_{D_2}$ for suitable $9$-dimensional subalgebras algebras $D_1$ and $D_2$ of $A$, containing $L$ as a maximal subfield (Proposition \ref{d1d2}), where, for a subalgebra $S\subset A,~G_S$ denotes the subgroup of $G$ consisting of automorphisms fixing $S$ pointwise.  
Let $\phi\in G(k)$ be irregular. We show that there is a $9$-dimensional subalgebra $B\subset A$ such that $\phi(B)=B$ (Theorem \ref{irreg}). 
We prove that every automorphism of an Albert division algebra $A$ fixes a cubic subfield of $A$ (Corollary \ref{fixedpointcor}) pointwise. 
If the fixed point subalgebra $L=A^{\phi}$ for an automorphism $\phi$ of $A$ 
is a noncyclic cubic extension of $k$, we prove $\phi$ is regular (Proposition \ref{regular}).  

Let $G$ be a semisimple, simply connected algebraic group, defined and isotropic over a field $k$. Let $G(k)$ denote the group of $k$-rational points of $G$ and 
$G(k)^+$ be the normal subgroup of $G(k)$, generated by the $k$-rational points of the unipotent radicals of the parabolic $k$-subgroups of $G$. The {\bf Kneser-Tits problem} asks if the quotient $W(k, G)=G(k)/{G(k)}^+$ is trivial (see \cite{PRap}, 7.2 and \cite{G} for a latest survey). 
In this paper we answer this question in the affirmative for a certain rank $2$ form of $E_8$. The form of $E_8$ that interests us has Tits index $E^{78}_{8,2}$, has anisotropic kernel a strict inner form of $E_6$. Such groups exist over $k$ if $k$ admits central division algebras of degree $3$ with nonsurjective reduced norm map or such algebras with unitary involutions over a quadratic extension of $k$ (see \cite{T1}). The Kneser-Tits problem for the form of $E_8$ mentioned above, has a reformulation in terms of the structure group of an Albert division algebra. This reformulation is due to J. Tits and R. Weiss, which made its first appearance in their book on Moufang polygons (see \cite{TW}, 37.41, 37.42 and page 418), in the form of a conjecture, henceforth referred to as the {\bf Tits-Weiss conjecture}. The assertion of the conjecture makes sense for {\it reduced} Albert algebras as well and follows rather easily from known results of Jacobson (see Theorem \ref{reduced}). Albert algebras are of profound importance in Lie theory and the theory of buildings. Exceptional groups of type $F_4, E_6, E_7$ and $E_8$ have close links with Albert algebras 
(see \cite{FF}, \cite{J1}, \cite{SV}, \cite{TW}). 
For example, over a 
base field $k$, all groups of type $F_4$ arise as groups of automorphisms of Albert algebras. Every Albert algebra $A$ over a field $k$ comes equipped with a cubic form called the {\bf norm form} of $A$. Also $A$ has a linear form defined on it, called its {\bf trace form}. Let $A$ be an Albert algebra over $k$, 
$\mathcal{N}$ be its norm and $\mathcal{T}$ be its trace. Recall that a similarity of $\mathcal{N}$ is a bijective linear map $\theta:A\rightarrow A$ 
such that $\mathcal{N}(\theta(x))=\alpha\mathcal{N}(x)$ for all $x\in A$, where $\alpha\in k^*$ is independent of $x$, 
called the factor of similitude of $\theta$.    
It is well known that the group of isometries of the norm form of $A$ is a simply connected group of type $E_6$ defined over $k$ (see for example \cite{SV}). The isotropy subgroup of $1\in A$ in the group of isometries is precisely the group of automorphisms of $A$.  
The conjecture of J. Tits and R. Weiss is about the group of similarities of the norm form $\mathcal{N}$ on $A$. It is clear that scalar multiplications 
$\mathcal{R}_t,~t\in k^*$ are norm similarities, since $\mathcal{N}(\mathcal{R}_t(x))=\mathcal{N}(tx)=t^3x$ for all $x\in A$. The right multiplications $\mathcal{R}_a,~a\in A^*$, by elements of $A$, however, are not similarities of the norm in general. The linear operators $U_a$ for $a\in A$ defined by 
$U_a=2\mathcal{R}_a^2-\mathcal{R}_{a^2}$ (these have a valid expression in characteristic $2$ as well, see Section 5) are norm similarities with similitude factor $\mathcal{N}(a)^2$ if $a$ is invertible in $A$. For an invertible element $u\in A$, the $u$-isotope of $A$ is the algebra $A^{(u)}$ whose underlying vector space structure is the same as $A$ but the multiplication on $A^{(u)}$ is defined as $x_uy=\{xuy\}$, where $\{xyz\}$ is the Jordan triple product in $A$, given by 
$$\{xyz\}=U_{x,z}(y),~~\\
U_{x,z}=\mathcal{R}_x\mathcal{R}_z+\mathcal{R}_z\mathcal{R}_x-\mathcal{R}_{xz},$$
$\mathcal{R}_x$ denotes the right multiplication on $A$ by $x$. Note that $U_a=U_{a,a}$. Two Albert algebras $A$ and $A'$ are said to be 
{\bf isotopic} if there is an isomorphism $\phi:A^{(u)}\longrightarrow A'$ for some invertible element $u\in A$. Such an isomorphism is called an {\bf isotopy} 
from $A$ to $A'$. Isotopy is an equivalence relation on the class of Albert algebras (see \cite{J1}, Sect. 7, Chap. VI). The set of all isotopies from $A\longrightarrow A$ forms a group 
under composition of maps, called the {\bf structure group} of $A$, denoted by $Str(A)$. It is known that isotopies are same as norm similarities as long as the base field is not too small (see \cite{J1}, Chap.VI, Theorem 6, Theorem 7). The subgroup of $Str(A)$ generated by the $U_a$, $a$ invertible, is called the {\bf inner structure group} of $A$, denoted by $Instr(A)$. One can prove that $Instr(A)$ is a normal subgroup of $Str(A)$ (see \cite{TW}, 37.42). Let $C$ be the subgroup of $Str(A)$ generated by the $\mathcal{R}_t,~t\in k^*$. Then $C.Instr(A)$ is normal in $Str(A)$. We can now state the conjecture of Tits and Weiss (see \cite{TW}, page 418):\\
\noindent
{\bf Conjecture (Tits-Weiss):} Let $A$ be an Albert division algebra over a field $k$. With the notation introduced above,
$$\frac{Str(A)}{C.Instr(A)}=\{1\}.$$
This conjecture has origins in the theory of Moufang polygons and is of fundamental importance to the subject. In this paper, 
we settle this conjecture for Albert division algebras which are pure first Tits constructions, thereby solving the Kneser-Tits problem for the $k$-forms of $E_8$ arising from such Albert division algebras. The elements of $Aut(A)\cap Instr(A)$ are called {\bf inner automorphisms} of $A$. We prove that every automorphism of an Albert division algebra which is a pure first Tits construction, is inner (Corollary \ref{pureinner}). 
The strategy to prove the above conjecture is to reduce the assertion about norm similarities to an assertion about automorphisms. 
The main step is to prove that every automorphism is inner. The key tool is a fixed point theorem for automorphisms of Albert algebras. We prove that every automorphism of an arbitrary Albert algebra $A$, fixes a nonzero element with trace zero. When the Albert algebra is a division algebra, this proves that every automorphism fixes a cubic subfield pointwise. We reprove a theorem of Jacobson on automorphisms of Cayley algebras and then use the same idea to derive the result on 
automorphisms of Albert division algebras. The analogy of Cayley algebras with Albert algebras is well known by the work of Petersson and Racine on 
Tits processes; our results strengthen this further. We apply our results on automorphisms of Albert division algebras to prove the triviality of yet another group, related to 
$R$-equivalence. Let $A$ be an Albert division algebra over $k$ which is a pure first construction and $G={\bf Aut}(A)$. We prove that $G(k)/R=1$, where $R$-stands for $R$-equivalence in algebraic groups (see \cite{M}).\\
\vskip1mm
\noindent
{\bf Groups with index $E^{78}_{8,2}$ :} Let $G$ be a semisimple simply connected group defined over $k$ with Tits index $E^{78}_{8,2}$ (see \cite{T3} for the index notation) and anisotropic kernel 
a strict inner form of $E_6$. Let $\Gamma$ denote the building associated to $G$, in the sense of (\cite{T2}) or (\cite{TW}). It is shown in (\cite{TW}) 
that $\Gamma$ is a Moufang hexagon, defined by a hexagonal system of type $27/F$, where $F/k$ is either a quadratic extension of $k$ or $F=k$. Recall from 
(\cite{TW}, 42.3.5, 42.6)  that such forms of $E_8$ exist if and only if $k$ admits Albert division algebras. Albert division algebras over $k$ exist if and only if there are degree $3$ central division algebras over $k$ with nonsurjective reduced norm maps, or $k$ admits degree $3$ central division algebras with unitary involution, over a quadratic extension of $k$, with nonsurjective reduced norm map. Hence these forms of $E_8$, for example, do not exist over finite fields, number fields, local fields or the field of real numbers, however the rational function field $\mathbb{Q}(t)$ admits these groups. 
In this context the following theorem of Albert is of importance (see \cite{J1}, Theorem 21, Chap. IX):
\begin{theorem} Let $k$ be a field such that there are degree $3$ central division algebras over $k$. Then there exist Albert division algebras over $k(t)$, the function field in one indeterminate. 
\end{theorem}
Contrast this with the fact that $k=\mathbb{C}((x_1,\cdots,x_n)),~n\geq 3$, admits Albert division algebras (see \cite{PR1}), even though there are no central division algebras of degree $3$ over $\mathbb{C}$. \\
\noindent
{\bf Moufang hexagons of type $27/F$ :} Moufang hexagons of type $27/F$ correspond to Albert division algebras over $k$. Let $\Gamma$ be the Moufang hexagon associated to an Albert division algebra $A$ over $k$, as described in Chapter 15 of (\cite{TW}), or 4.7 of (\cite{T2}); see also 
(\cite{TH}) for an excellent introduction to the subject. Let $G$ be the group of ``linear'' automorphisms of $\Gamma$ (see \cite{TW} or \cite{W1} or \cite{W2}) and let $G^{\dagger}$ be the subgroup generated by the root groups of $\Gamma$. Then $G^{\dagger}$ is a normal subgroup of $G$ and the quotient $G/G^{\dagger}$ is isomorphic to $H/H^{\dagger}$, where $H$ is the structure group of the Albert algebra $A$ and $H^{\dagger}$ is the subgroup of $G$ generated by all maps of the form $x\mapsto tx,~t\in k^*$ or $x\mapsto U_b(x),~b\in A^*$; this is proved in 37.41 of (\cite{TW}). 
The group $G$ is the group of $k$ rational points of a an algebraic group defined over $k$ with index $E^{78}_{8,2}$, with anisotropic kernel $H$, 
the structure group of $A$. By 42.3.6 of (\cite{TW}), the root groups of $\Gamma$ are the groups $\mathcal{U}_{\alpha}(k)$, where $\alpha$ is a nondivisible root relative to a maximal $k$-split torus $S$ and $\mathcal{U}_{\alpha}$ is the corresponding unipotent $k$-group. This explains the Kneser-Tits problem in the context of the Tits-Weiss conjecture stated in (\cite{TW}, page 418) for the structure group of $A$.\\
\noindent
{\bf Structure of the paper :}
In the paper, the results that are needed to settle the Tits-Weiss conjecture, mostly from Section 5, are proved over arbitrary characteristics (one can work with cubic norm structures and use the formulae derived), however some other results on automorphisms in Section 6 may need omission of characteristic $2$ and $3$. For this reason and for simplicity of exposition, we choose to take the algebra point of view versus the cubic norm structure approach for Albert algebras and assume that characteristic of the base field is different from $2$ and $3$. We now describe the paper briefly. In Section 2, we introduce the notion of an Albert algebra and describe Tits constructions of such 
algebras, which will be used in the sequel. Section 3 gives preliminary material on Moufang hexagons of type $27/F$ and relates the Kneser-Tits problem to the Tits-Weiss conjecture. In Section 4, we take up the study of automorphisms of Albert algebras. The main tool in this is the fixed point subalgebra of an automorphism, or more generally, of a given (algebraic) subgroup of the algebraic group of automorphisms ${\bf Aut}(A)$ of a given Albert algebra $A$ over $k$. The results of Section 5 are oriented towards a proof of the Tits-Weiss conjecture and are valid over arbitrary characteristics. This section depends very little (except for the fixed point theorem) on Section 4. The key result being that every automorphism of an Albert division algebra, which is a pure first construction, is inner (i.e. is in the inner structure group of $A$ as defined in the introduction). This is proved step by step, while proving some results valid for general first constructions. We use the theorem of Wang on the reduced Whitehead group of division algebras (\cite{W}, \cite{T4}, Proposition 2.7) to settle the Tits-Weiss conjecture. We also derive some results for the norm similarities of first constructions, which help us in the proof. Using an analogue of Wang's result for the unitary Whitehead group of a division algebra, due to Yanchevskii, we prove that some automorphisms of a second Tits construction Albert algebra are inner (Cor. 5.2). We come back to the study of automorphisms in Section 6. We prove that irregular automorphisms of Albert division algebras are rather robust, we prove that such automorphisms stabilize $9$-dimensional subalgebras. We make some remarks on regular automorphisms as well. It seems likely that any automorphism is a product of automorphisms stabilizing $9$-dimensional subalgebras. The results of this section are general in 
nature, along the theme of Section 5. We apply the results of Section 5 to the $R$-triviality problem for groups of type $F_4$ over $k$, arising from pure first construction Albert division algebras. We conclude with some remarks, which may help in attacking the problem in its full glory.
\section{\bf Albert algebras}\label{albert}
In this section, we introduce some basic material on Albert algebras over a field $k$, assumed to be of characteristic not $2$ or $3$ for simplicity of exposition. As remarked before, one may work with cubic norm structures instead of algebras and the results continue to hold true. We assume some familiarity with the notion of an octonion (Cayley) algebra, which can be found in (\cite{SV}) for example. For a detailed account of Albert algebras and octonion algebras we refer to (\cite{J1}), (\cite{SV}) or (\cite{KMRT}).\\
\vskip1mm
\noindent
{\bf Reduced Albert algebras:} Let $C$ be an octonion algebra over $k$ and let $\Gamma=Diag(\gamma_1,\gamma_2,\gamma_3)\in GL(3,k)$ be a diagonal matrix. 
Let $x\mapsto\overline{x}$ denote the standard involution on $C$. Let $*:M_3(C)\rightarrow M_3(C)$ denote the involution $X\mapsto X^*=\Gamma^{-1}\overline{X}^t\Gamma$, where $M_3(C)$ denotes the (nonassociative) ring of $3\times 3$ matrices with entries from $C$ and $\overline{X}=(\overline{x_{ij}})$ for $X=(x_{ij})$. Let 
$\mathcal{H}(M_3(C),*)=\mathcal{H}_3(C,\Gamma)$ denote the space of hermitian (i.e. $*$-symmetric) matrices in $M_3(C)$. Then clearly $Dim_k(\mathcal{H}_3(C,\Gamma))=27$. We define a product on this space by $X\cdot Y=\frac{1}{2}(XY+YX)$. We thus get a $k$-algebra structure, called a {\it reduced} Albert algebra over $k$. These algebras have zero divisors. We call $A=\mathcal{H}_3(C,\Gamma)$ {\it split} if $C$ is split over $k$. It turns out that there is a unique split Albert algebra over a field $k$, up to isomorphism. An Albert algebra over $k$ can be defined as a $k$-algebra $A$ such that $A$ is isomorphic to the split Albert algebra over some field extension of $k$. An Albert algebra in which all non-zero elements are invertible, i.e. have their $U$-operators bijective, is called an Albert {\it division} algebra. The automorphism group of the split Albert algebra over $k$ is the split form of $F_4$ over $k$ (\cite{KMRT}, 25.13). The automorphism group of an Albert algebra is a $k$-anisotropic form of $F_4$ if and only if the algebra has no nonzero nilpotents (see \cite{PR5}, page 205). In particular, the automorphism group of an Albert division algebra is a $k$-anisotropic form of $F_4$. For a reduced Albert algebra, there is a natural notion of the {\bf trace}, namely, the trace of a matrix in a reduced Albert algebra $A$ is the sum of its diagonal entries (which are scalars) and is a 
linear form on $A$. Similarly, there is a {\bf norm} defined on $A$, which is the ``determinant'' computed in a suitable way, and is a cubic form on $A$. 
The notion of trace and norm passes down to an arbitrary Albert algebra over $k$. Tits has given two general constructions of all Albert algebras over a field $k$. It is known that any Albert algebra over $k$ comes from either of the two constructions. We shall give the norm and trace maps in the two constructions below. We now describe them briefly. \\
\noindent
{\bf Tits constructions of Albert algebras}\\
\noindent
{\bf The first construction:}\\
Let $k$ be a base field as fixed before. Let $D$ be a central simple (associative)  algebra over $k$ of degree $3$. We will denote by $D_+$ the (special) Jordan algebra structure on $D$, with multiplication $x\cdot y=\frac{1}{2}(xy+yx),~x,y\in D$. Let $\mu\in k^*$ be a 
scalar. To this data, one associates an Albert algebra $J(D,\mu)$ as follows:
$$J(D,\mu)=D_0\oplus D_1\oplus D_2,$$
where $D_i,~i=0,1,2$, is a copy of D. For the multiplication on $J(D,\mu)$, we need more notation. For $a,b\in D$ define
$$a\cdot b=\frac{1}{2}(ab+ba),~a\times b=2a\cdot b-t(a)b-t(b)a+(t(a)t(b)-t(a\cdot b))$$and $\widetilde{a}=\frac{1}{2}(t(a)-a)$, 
where $t:D\longrightarrow k$ is the reduced trace map of $D$. The multiplication on $J(D,\mu)$ is given by the formula :
$$(a_0,a_1,a_2)(b_0,b_1,b_2)~~~~~~~~~~~~~~~~~~~~~~~~~~~~~~~~~~~~~~~~~~~~~~~~~~~~~~~~~~~~~~~~~~~~~~~~~~~~~$$
$$=(a_0\cdot b_0+\widetilde{a_1b_2}+\widetilde{b_2a_2}, \widetilde{a_0}b_1+\widetilde{b_0}a_1+(2\mu)^{-1}a_2\times b_2,a_2\widetilde{b_0}
+b_2\widetilde{a_0}+\frac{1}{2}\mu a_1\times b_1).~~~~~~~~~~~$$ 
It is known that $J(D,\mu)$ is an Albert algebra over $k$ (see \cite{KMRT} or \cite{SV} for more details). Further, $J(D,\mu)$ is a {\bf division algebra} if and only if $\mu$ is not a reduced norm from $D$. Clearly $D_+$ is a subalgebra of $J(D,\mu)$. Let $A$ be an Albert algebra over $k$ and let $D_+\subset A$ for some degree $3$ central simple algebra $D$. Then there exists $\mu\in k^*$ such that $A\simeq J(D,\mu)$. \\
\noindent
{\bf Trace and Norm maps :}Let $A=J(D,\mu)$ be as defined above. Let $t=T_D$ and $n=N_D$ denote the reduced trace and reduced norm maps on $D$ (see \cite{J3}). The trace 
$T$ and the norm $N$ on $A$ are given by the formulae:
$$T(x,y,z)=t(x),~~N(x,y,z)=n(x)+\mu n(y)+\mu^{-1}n(z)-t(xyz).$$
From this, one gets an expression for the {\it trace bilinear form} on $A$, defined by, $T(x,y)=T(xy),~x,y\in A$. Therefore, one has, for $x=(x_0,x_1,x_2),~y=(y_0,y_1,y_2)$,  
$$T(x,y)=t(x_0y_0)+t(x_1y_2)+t(x_2y_1).$$
One knows that an Albert algebra $A$ is a division algebra if and only if its norm form is anisotropic over $k$. \\
\noindent
{\bf The adjoint map :}Let $A=J(D,\mu)$. One defines a quadratic map $\#:A\rightarrow A$ by $x\mapsto x^{\#}$, called the {\it adjoint} of $x$. Let $x=(x_0,x_1,x_2)$.  We define 
$$x^{\#}=(x_0^{\#}-x_1x_2,\mu^{-1}x_2^{\#}-x_0x_1,\mu x_1^{\#}-x_2x_0),$$
where, for $y\in D, 2y^{\#}=y\times y$. One can prove that $xx^{\#}=x^{\#}x=N(x)$ for all $x\in A$(see \cite{J1} for details). 
\vskip1mm
\noindent
{\bf The second construction:}\\
Let $K/k$ be a quadratic extension and let $(B,\tau)$ be a central simple $K$-algebra of degree $3$ over $K$ with a unitary involution $\tau$ over $K/k$. Let $u\in B^*$ be such that $\tau(u)=u$ and $N(u)=\mu\overline{\mu}$ for some $\mu\in K^*$, here bar denotes the nontrivial $k$-automorphism of $K$ and $N$ is the reduced norm map of $B$. Let $\mathcal{H}(B,\tau)$ be the special Jordan algebra structure on the $k$-vector subspace of $B$ of $\tau$-symmetric elements in $B$, with multiplication as in $B_+$. Let $J(B,\tau,u,\mu)=\mathcal{H}(B,\tau)
\oplus B$. With the notation introduced above, we define a multiplication on 
$J(B,\tau,u,\mu)$ by 
$$(a_0,a)(b_0,b)=(a_0\cdot b_0+\widetilde{au\tau(b)}+\widetilde{bu\tau(a)},
\widetilde{a_0}b+\widetilde{b_0}a+\overline{\mu}(\tau(a)\times\tau(b))u^{-1}).$$
Then $J(B,\tau,u,\mu)$ is an Albert algebra over $k$ and is a division algebra if and only if $\mu$ is not a reduced norm from $B$. Clearly $\mathcal{H}(B,\tau)$ is a subalgebra of $J(B,\tau,u,\mu)$. It is known that if $\mathcal{H}(B,\tau)$ is a subalgebra of an Albert algebra $A$ over $k$, then there are suitable parameters $u\in B^*$ and $\mu\in K^*$, where $K$ is the centre of $B$, such that $A\simeq J(B,\tau,u,\mu)$. \\
\noindent
{\bf Trace and Norm maps :}Let $A=J(B,\tau,u,\mu)$ be an Albert algebra arising from the second construction. Let $t=t_B$ and $n=N_B$ denote the reduced trace and reduced norm maps on $B$ (see \cite{J3}). The trace $T$ and the norm $N$ on $J(B,\tau,u,\mu)$ are given by the formulae:
$$T(b_0,b)=t(b_0),~~N(b_0,b)= n(b_0)+\mu n(b)+\overline{\mu}n(\tau(b))-t(b_0bu\tau(b)).$$ 
From this, one gets an expression for the trace bilinear form on $J(B,\tau,u,\mu)$, defined by $T(x,y)=T(xy),~x,y\in A$. Therefore, we have, for $x=(a_0,a),
y=(b_0,b)$,
$$T(x,y)=t(a_0b_0)+t(au\tau(b))+t(u\tau(a)b).$$
The Albert algebra $A=J(B,\tau,u,\mu)$ is a division algebra if and only if the norm form is anisotropic over $k$. \\
\noindent
{\bf The adjoint map :} Let $A=J(B,\tau,u,\mu)$ and $x=(a_0,a)$. In this case, the adjoint map is given by 
$$x^{\#}=(a_0^{\#}-au\tau(a),\overline{\mu}\tau(a)^{\#}u^{-1}-a_0a),$$
where for $y\in B,~2y^{\#}=y\times y$ as defined above. One has $N(x)=xx^{\#}=x^{\#}x$ for all $x\in A$. 
\vskip1mm
\noindent
{\bf Remarks :}It is known that all Albert algebras arise from these two constructions (see \cite{KMRT}) and these are not mutually exclusive. 
Hence there are Albert algebras which are of mixed type as well as of pure type. For simple recipe for pure second constructions, see (\cite{Th}) and for examples of pure first constructions, see (\cite{PR1}). Note that if $A$ is a pure first construction Albert division algebra then every $9$ dimensional subalgebra of $A$ must necessarily be of the form $D_+$ for a degree $3$ central division algebra $D$ over $k$. There is a cohomological characterization of pure second construction Albert algebras. An Albert algebra $A$ is a pure second construction if and only if $f_3(A)\neq 0$ (see \cite{KMRT}). However, such a characterization for pure first construction Albert algebras doesn't seem to be available in the literature. It is well known that any cubic subfield of an Albert division algebra reduces it, i.e., if $L\subset A$ is a cubic subfield, where $A$ is a division algebra over $k$, then $A\otimes L$ is reduced over $L$. Moreover, when $A$ is a first construction, every cubic subfield is a splitting field for $A$ (see \cite{PR2}). It can be shown that if $A$ is an Albert division algebra, ${\bf Aut}(A)$ remains anisotropic over any extension of $k$ of degree coprime with $3$.\\ 
\noindent
{\bf Pure first Tits construction Albert algebras :}Albert algebras which arise through the first Tits construction and cannot arise from the second construction are called pure first constructions. Let $A$ be an Albert algebra over $k$ and let $S\subset A$ be a $9$-dimensional subalgebra. It is known that $S=D_+$ or $S=\mathcal{H}(B,\tau)$ for $D$ a degree $3$ central simple algebra over $k$ or $B$ a degree $3$ central simple algebra over a quadratic extension $K/k$ with a unitary involution $\tau$. If $A$ is a pure first construction, then every $9$-dimensional subalgebra $S$ must be of the form $D_+$, by remarks made before on $9$ dimensional subalgebras of Albert algebras. Let $k_0$ be an algebraically closed field of characteristic not $2$ or $3$ and $n$ a positive integer. Then, by the work of Petersson and Racine (\cite{PR1}), every Albert division algebra over the iterated Laurent series field $k=k_0((x_1,\cdots,x_n)),~n\geq 3$, is a pure first construction. 

Following characterisation of pure first Tits construction Albert division algebras was communicated to us by Holger P. Petersson:
\begin{theorem}For an Albert division algebra over a field $k$ of arbitrary characteristic, the following conditions are equivalent.
\begin{enumerate}
\item[(i)] $A$ is a pure first Tits construction.
\item[(ii)] $A$ is a first Tits construction and every separable cubic subfield of $A$ is cyclic.
\item[(iii)] Every separable cubic subfield of any isotope of $A$ is cyclic. 
\end{enumerate}
\end{theorem}
\begin{proof} (i) implies (ii). Suppose $A$ is a pure first Tits construction and let $L\subset A$ be a separable cubic subfield. By (\cite{PR2}), Cor. 4.5, there exists $\lambda\in k^*$ such that the inclusion $L\hookrightarrow A$ extends to an embedding from the first Tits construction $J=J(L,\lambda)$ to $A$. By the hypothesis, $J$ has the form $D_+$ for some central associative division algebra $D$ of degree $3$ over $k$. But by (Prop. 5.1, \cite{PR6}), this forces the cubic extension $L/k$ to be cyclic.\\
\noindent
(ii) implies (iii). If $A$ is a first Tits construction, then by (\cite{PR2}, Cor. 4.9) all its isotopes are isomorphic. Hence (iii) is a trivial consequence 
of (ii).\\
\noindent
(iii) implies (i). Arguing indirectly, let us suppose that (iii) holds, but $A$ is not a pure first Tits construction. Then there exists a central associative division algebra $(B,\tau)$ of degree $3$ with an involution of second kind over $k$ such that $\mathcal{H}(B,\tau)$ is a subalgebra of $A$ and $K$, the center of $B$, is a separable quadratic field extension of $k$. Passing to an appropriate isotope of $\mathcal{H}(B,\tau)$, equivalently, replacing $\tau$ by $Int(u)\circ\tau$ for some invertible $u\in \mathcal{H}(B,\tau)$, we may assume that the involution $\tau$ is distinguished (see \cite{P}, Theorem 2.10). But then, by loc.cit. Theorem 3.1, $\mathcal{H}(B,\tau)$, hence $A$, contains a separable cubic subfield with discriminant algebra $K/k$, which therefore cannot be cyclic.
\end{proof}
\noindent
{\bf The structure group of Albert algebras :}Recall that every Albert algebra comes equipped with a cubic form $\mathcal{N}$, called the norm. The isometries of 
$\mathcal{N}$ form the $k$-rational points of a simply connected $k$-algebraic group of type $E_6$. This group contains the algebraic group ${\bf Aut}(A)$ of automorphisms of $A$, we denote its group of $k$-rational points by $Aut(A)$. The group of similarities of $\mathcal{N}$ is called the {\bf structure group} of $A$ and we denote it by $Str(A)$. This coincides with the group of $k$-rational points of a strict inner $k$-form of $E_6$, which we may denote by ${\bf Str}(A)$. Let $a\in A$ and $\mathcal{R}_a$ denote the right multiplication by $a$ acting on $A$. We let $U_a=2\mathcal{R}_a^2-\mathcal{R}_{a^2}$. Then $U_a\in Str(A)$ for all invertible $a\in A$. The (normal) subgroup generated by $U_a,~a$ invertible, is called the {\bf Inner structure group} of $A$ and denoted by $Instr(A)$, this also is the group of $k$-points of a certain algebraic group 
${\bf Instr(A)}$. The group $Aut(A)\cap Instr(A)$ is called the group of {\bf inner automorphisms} of $A$. Finally, recall that a norm isometry is an automorphism if and only if it fixes the identity element of $A$ (see \cite{J1}, Section 7, Chapter VI). 
\section{\bf Moufang Hexagons of type $27/F$}
In this section, we explain the construction of Moufang hexagons of type $27/F$ and connect them to the Kneser-Tits problem as well as the Tits-Weiss conjecture 
mentioned in the introduction. We reproduce below some of the material from (\cite{TH}) for our exposition. Let $A$ be an Albert division algebra over $k$. Let $U_1, U_3, U_5$ be three groups isomorphic to the additive group $(A,+)$ and let $U_2,U_4,U_6$ be three groups isomorphic to the additive group $(k,+)$. Let $x_i$ denote the isomorphism of 
$(A,+)$ or $(k,+)$ with $U_i$. We define a group $U_+$, generated by the $U_i$ subject to the commutation relations as follows (see \cite{TW} 8.13): 
$$[U_1,U_2]= [U_2,U_3]=[U_3,U_4]=[U_4,U_5]=[U_5,U_6]=1,$$
$$[U_2,U_4] = [U_4,U_6]=1,$$
$$[U_1,U_4]  =  [U_2,U_5]=[U_3,U_6]=1,$$
$$[x_1(a),x_3(b)]  =  x_2(T(a,b)),$$
$$[x_3(a),x_5(b)] = x_4(T(a,b)),$$
$$[x_1(a),x_5(b)]  =  x_2(-T(a^{\#},b))x_3(a\times b)x_4(T(a,b^{\#})),$$
$$[x_2(t),x_6(u)]  =  x_4(tu),$$
$$[x_1(a),x_6(t)]  =  x_2(-tN(a))x_3(ta^{\#})x_4(t^2N(a))x_5(-ta),$$
for all $a,b\in A$ and all $t,u\in k$. We construct a graph $\Gamma$ from this data as follows: Let $\phi$ be a map from $\{1,2,\cdots,6\}$ to the set of subgroups of $U_+$ defined by :
$$\phi(i)=U_{[1,i]},~1\leq i\leq 3,~~\phi(i)=U_{[i-3,3]}, 4\leq i\leq 6,$$
where 
$$ U_{[i,j]}=<U_i, U_{i+1},\cdots, U_j>,~~i\leq j< i+n;~~U_{[i,j]}=1~\rm{otherwise}.$$
Let the vertex set of $\Gamma$ be defined by 
$$V(\Gamma)=\{(i,\phi(i)g|1\leq i\leq 6,~g\in U_+\},$$
where $\phi(i)g$ is the right coset of $\phi(i)$ containing $g$. The edge set of $\Gamma$ is defined by 
$$E(\Gamma)=\{((i,R),(j,S))|~|i-j|=1,~R\cap S\neq\emptyset\},$$
here $|i-j|$ is computed modulo $6$. This gives a graph $\Gamma=(V(\Gamma), E(\Gamma))$, which is completely determined (up to isomorphism) by the $4$-tuple $(U_+,U_1,U_2,U_3)$. The graph $\Gamma$ is the building associated to the $k$-algebraic group with index $E^{78}_{8,2}$ and anisotropic kernel the strict inner 
$k$-form of $E_6$ corresponding to the structure group of $A$. This graph is called a Moufang hexagon of type $27/F$, where $F=k$ if $A$ is a first construction and $F=K/k$, a quadratic extension of $k$ if $A$ is a second construction. Let $G$ be the group of ``linear'' automorphisms of $\Gamma$. The linear automorphisms of $\Gamma$ are those which are induced by the linear elements of the structure group of the Albert algebra $A$. Then $G$ is the group of $k$-rational points of the 
$k$-form of $E_8$ described above. Let $G^{\dagger}$ denote the subgroup of $G$ generated by the ``root groups'' of $\Gamma$ (see \cite{TW} for definition). Then 
$G/G^{\dagger}\simeq H/H^{\dagger}$, where $H$ is the structure group of $A$ and $H^{\dagger}$ the subgroup of $H$ generated by the maps $\mathcal{R}_t,t\in k^*$ and $U_a,~a\in A^*$ (see 37.41 of \cite{TW}). By (\cite{TW}, 42.3.6), the root groups of $\Gamma$ are the groups $\mathcal{U}_{\alpha}(k)$, where $\alpha$ is a nondivisible root corresponding to a maximal split $k$-torus $S$ in the relative rank $2$ form of $E_8$ described above and $\mathcal{U}_{\alpha}$ is the unipotent group corresponding to $\alpha$. Hence if ${\bf G}$ denote the $k$-form of $E_8$ in question, with ${\bf G}(k)=G$, we have 
$${\bf G}(k)/{\bf U}(k)\simeq G/G^{\dagger}\simeq H/H^{\dagger},$$
where ${\bf U}(k)$ denotes the subgroup of ${\bf G}(k)=G$ generated by the $k$-rational points of the unipotent radicals of parabolic $k$-subgroups of ${\bf G}$.  The Kneser-Tits problem is to determine if the quotient ${\bf G}(k)/{\bf U}(k)$ is trivial. The Tits-Weiss conjecture asserts that the quotient $G/G^{\dagger}$ is trivial.   

\section{\bf Automorphisms of Albert algebras}\label{autos1}
In this section, we prove some results about automorphisms of Albert algebras in general. Some of these are intended for the proof of the Tits-Weiss conjecture, but others may be of independent interest. Let $A$ be an Albert algebra over $k$ and let $A_0$ denote the subspace of elements of $A$ having zero trace. We shall denote the algebraic group of automorphisms of $A$ by ${\bf Aut}(A)$ and $Aut(A)$ will denote its group of $k$-rational points. Also, for a subalgebra $S\subset A$, 
$G_S={\bf Aut}(A/S)$ will denote the algebraic subgroup of ${\bf Aut}(A)$ consisting of automorphisms fixing $S$ pointwise. For a central division algebra $D$ over $k$, we will denote by ${\bf SL}(1,D)$ the algebraic group of norm $1$ elements in $D$ and $SL(1,D)$ will denote its group of $k$ rational points. We need descriptions of certain subgroups of ${\bf Aut}(A)$ for our purpose, we include it here (see \cite{KMRT}), Chapter IX, Section 39). We have,
\begin{proposition} Let $A=\mathcal{H}_3(C,\Gamma)$ be a reduced Albert algebra over $k$ and let $H\subset {\bf Aut}(A)$ be the algebraic subgroup of ${\bf Aut}(A)$ consisting of all automorphisms of $A$ which fix the three diagonal idempotents in $A$. Then $H$ is a $k$-subgroup of type $D_4$, in fact $H\simeq {\bf Spin}(n_C)$, where $n_C$ is the norm on the octonion algebra $C$. 
\end{proposition}
We will denote the subgroup in the proposition by ${\bf Spin}(8)$ when the norm form $n_C$ of $C$ is not important in the discussion. 
\begin{corollary} Let $A$ be an Albert algebra over $k$ and $L\subset A$ be a cubic \'{e}tale subalgebra of $A$. Then ${\bf Aut}(A/L)$ is a $k$-subgroup of ${\bf Aut}(A)$ of type $D_4$. 
\end{corollary} 
\begin{proposition}Let $A$ be an Albert algebra arising from the first construction. Let $D_+\subset A$. Then, with the notation introduced above, 
${\bf Aut}(A/D_+)$ is a $k$-subgroup of ${\bf Aut}(A)$, isomorphic to ${\bf SL}(1,D)$ over $k$.
\end{proposition}
\noindent
We now can prove our first key result, the proof rests on the conjugacy theorem for maximal tori in algebraic groups. Recall that for an Albert algebra $A$ the subspace of trace zero elements is denoted by $A_0$. 
\begin{theorem}{\label{fixedpoint}}Let $\phi\in Aut(A)$. There exists a nonzero element $x_0\in A_0$ such that $\phi(x_0)=x_0$. 
\end{theorem}
\begin{proof}Since $\phi$ is an automorphism, we have $\phi(A_0)=A_0$. To prove the assertion, it suffices to prove that $1$ is an eigenvalue for the restriction of $\phi$ to $A_0$. To prove the assertion on eigenvalues of $\phi$, we may assume $k$ is algebraically closed. We may, by considering the Jordan decomposition of $\phi$, assume that $\phi$ is semisimple. We can write $A=\mathcal{H}_3(\mathcal{C},1)$, for the split Cayley algebra $\mathcal{C}$ over $k$. Let $T\subset G={\bf Aut}(A)$ be a maximal torus containing $\phi$. Recall that the 
subgroup $H$ of $G$ consisting of all automorphisms of $A$ that fix the three diagonal idempotents in $A$ is isomorphic to ${\bf Spin}(8)$ and hence has rank $4$. Therefore, for some $g\in G,g^{-1}Tg\subset H$. Hence $g^{-1}\phi g$ fixes the three diagonal idempotents. This implies that $g^{-1}\phi g$ fixes $a=diag(1,-1,0)\in A_0$. Hence $\phi$ fixes $x_0=g(a)\in A_0$.
\end{proof}
\begin{corollary}{\label{fixedpointcor}}Let $A$ be an Albert division algebra over a field $k$ of arbitrary characteristic and $\phi\in Aut(A)$ an automorphism of $A$. Then $\phi$ fixes a cubic subfield of $A$ pointwise.
\end{corollary}
\begin{proof} By the above theorem, there is a nonzero element $x_0\in A_0$ such that $\phi(x_0)=x_0$. Let $L$ be the subalgebra generated by $x_0$. Then $L$ is a cubic field extension of $k$ and clearly $\phi$ fixes $L$ pointwise.
\end{proof}
\begin{theorem}{\label{etalefixedpoint}}Let $A$ be an Albert algebra over $k$ , $char(k)\neq 2,3$, and suppose 
$\phi\in Aut(A)$ is semisimple. Then $\phi$ fixes a cubic \'{e}tale subalgebra $L\subset A$ pointwise. 
\end{theorem}
\begin{proof} Let $B=A^{\phi}$ denote the subalgebra of fixed points of $\phi$ in $A$. Then $B$ is compatible with arbitrary base field extensions, i.e. for any extension $M/k,~~(A\otimes_kM)^{\phi\otimes 1}=A^{\phi}\otimes_kM$.
By an argument as in the proof of the above theorem, $A^{\phi}\otimes_k\overline{k}$ contains a split cubic \'{e}tale subalgebra, where $\overline{k}$ is a fixed algebraic closure of $k$. In particular, the nil radical of $A^{\phi}\otimes_k\overline{k}$ has codimension at least $3$ over $\overline{k}$. By (\cite{PR4}), $N=$ Nil radical of $A^{\phi}$ coincides with the radical of the bilinear trace of $A^{\phi}$. Hence $\overline{N}=N\otimes_k\overline{k}=$Nilrad$(A^{\phi}\otimes_k\overline{k})$. Since $char(k)\neq 2,3$, the algebra $A^{\phi}/N$ is separable. Hence there exists a cubic separable Jordan subalgebra $S\subset A^{\phi}$ such that $A^{\phi}=S\oplus N$, direct sum of subspaces. But $dim(S)\geq 3$, hence $S$ contains a cubic \'{e}tale subalgebra $L$ and since $S\subset A^{\phi}$, $\phi$ fixes $L$ pointwise. 
\end{proof}
\noindent
{\bf Remark :} The original proof of this theorem had some mistakes, the proof above was suggested to us by Holger P. Petersson. 
\begin{corollary}Let $G$ be a group of type $F_4$, defined and anisotropic over a (perfect) field $k$, $char(k)\neq 2,3$. Let $T$ be a maximal torus in $G$, defined over $k$. Then there exists a simply connected subgroup $H\subset G$ defined over $k$, of type $D_4$, such that $T\subset H$ over $k$. 
\end{corollary}
\begin{proof} First assume that $k$ is an algebraic extension of a finite field. Then $Br(k)=0$ (see e.g. \cite{S}). Hence any Albert algebra over $k$ is split and every group of type $F_4$ over $k$ is split. Hence $k$ is not an algebraic extension of a finite field and we can apply (\cite{B}, Proposition 8.8) to find a regular element $\phi\in T(k)$. Let $A$ be the Albert algebra over $k$ such that $G={\bf Aut}(A)$. Since $G$ is anisotropic over $k$ and $k$ is perfect, by (\cite{R}, Proposition 6.3), all elements of ${\bf Aut}(A)(k)=Aut(A)$ are semisimple. Hence by the theorem above, there is a cubic \'{e}tale subalgebra $L\subset A$ fixed pointwise by $\phi$. Let $G_L={\bf Aut}(A/L)$ be the subgroup of $G$ consisting of automorphisms of $A$ that fix $L$ pointwise. Then $G_L$ is a simply connected group of type $D_4$, defined over $k$ (\cite{KMRT}, Chapter IX). Clearly $\phi\in G_L$. Let $T'\subset G_L$ be a maximal torus defined over $k$ such that $\phi\in T'$. But then $T'$ is also maximal in $G$. Since $\phi$ is regular, it belongs to a unique maximal torus. Hence $T=T'\subset G_L$. Hence we have proved the assertion.  
\end{proof}
\noindent
{\bf Some automorphisms of first Tits constructions :}\\
We recall some known automorphisms of Albert algebras which are Tits first constructions and introduce notation for them for our future use. The facts used on Albert algebras can be found in Chapter IX of (\cite{KMRT}) or Chapter IX of ({\cite{J1}).  
\vskip1mm
Let $A=J(D,\mu)$ be an Albert algebra and $a,b\in D^*$ be such that $N_D(a)=N_D(b)$. Let $\psi_{a,b}:A\longrightarrow A$ be the automorphism of $A$ given by 
$$\psi_{a,b}(x,y,z)=(axa^{-1},ayb^{-1},bza^{-1}).$$
Then clearly $\psi_{a,b}(D_0)=D_0$ and it is known (see \cite{KMRT}, Chapter IX) that any automorphism of $A$ stabilizing $D_0$ is of this form. A special case of this merits a separate mention. Let $p\in SL(1,D)$. Let $\mathcal{J}_p:A\longrightarrow A$ denote the automorphism $\psi_{1,p^{-1}}$, i.e. 
$$\mathcal{J}_p(x,y,z)=(x,yp,p^{-1}z).$$ 
We shall write $\mathcal{I}_a$ for the automorphism $\mathcal{I}_a(x,y,z)=(axa^{-1},aya^{-1},aza^{-1})$. Hence $\mathcal{I}_a=\psi_{a,a}$. Observe that 
$$\psi_{a,b}\psi_{c,d}=\psi_{ac,bd},~~\mathcal{I}_{ab}=\mathcal{I}_a\mathcal{I}_b,~~\mathcal{J}_{ab}=\mathcal{J}_b\mathcal{J}_a,~~a,b\in D^.$$
It is known that any automorphism of $A$ fixing $D_0$ pointwise is of this form. It is also known that the subgroup 
${\bf Aut}(A/D_+)$ of ${\bf Aut}(A)$ fixing $D_+$ pointwise is isomorphic to ${\bf SL}(1,D)$ (see \cite{KMRT}, Section 39.B). We record the extension of automorphisms of subalgebras of the form $D_+$ of $A$ below, since we will use this frequently in the paper: 
\begin{lemma}{\label{extendaut}}Let $A$ be an Albert algebra over $k$ and $D_+\subset A$. Then any automorphism of $D$ extends to an automorphism of $A$. Since every automorphism of $D$ is inner, we may extend any automorphism of $D$ to an automorphism by $\psi_{a,b}$, for suitable $a,b\in D^*$ and, in particular, $\mathcal{I}_a=\psi_{a,a}$ gives such an extension for a suitable $a\in D^*$. If 
$A=J(D,\mu)$, then every automorphism of $A$ stabilizing $D_0$ is of the form $\psi_{a,b}$ for suitable $a,b\in D^*$.
\end{lemma} 
\vskip1mm
\noindent
{\bf Automorphisms of second Tits constructions :}\\
Let $A$ be an Albert algebra over a field $k$. We will briefly discuss here automorphisms of $A$ that leave a $9$-dimensional subalgebra invariant. Let 
$A=J(B,\sigma,u,\mu)$ be a second Tits construction Albert division algebra. Let $\sigma'=Int(u)\circ\sigma$. We then have the special unitary group 
of $(B,\sigma')$ given by 
$$SU(B,\sigma')=\{x\in B|x\sigma'(x)=1,~N_B(x)=1\}.$$
Let $p\in SU(B,\sigma)$ and $q\in SU(B,\sigma')$. Let $\phi_{p,q}$ denote the automorphism of $A$ given by 
$$\phi_{p,q}(a,b)=(pa\sigma(p),pbq).$$
Then $\phi_{p,q}$ stabilizes $\mathcal{H}(B,\sigma)$ and it is known (see \cite{KMRT}, 39.16) that any automorphism of $A$ stabilizing $\mathcal{H}(B,\sigma)$ is of this form. Let $p\in SU(B,\sigma)$. We will denote $\phi_{p,1}$ simply by $\phi_p$. 
\vskip2mm
\noindent
{\bf Fixed points of subgroups of ${\bf Aut}(A)$:} Here we discuss the subalgebra of fixed points of certain subgroups of $G={\bf Aut}(A)$. 
For a subgroup $H\subset G$, we denote the subalgebra of $A$ of fixed points of $H$ by 
$$A^H=\{x\in A|h(x)=x~(h\in H)\}.$$
\begin{theorem}
Let $A$ be an Albert algebra over a field $k$, $char(k)\neq 2,3$. Let 
$T\subset G={\bf Aut}(A)$ be a maximal torus defined over $k$. Then $A^T$ is a cubic \'{e}tale subalgebra of $A$.
\end{theorem}
\noindent
We need the following lemma for the proof.
\begin{lemma} Let $\mathcal{C}$ be a split Cayley algebra over $k$ and let $A=\mathcal{H}_3(\mathcal{C},1)$. Then there exists a semisimple automorphism $\phi\in 
Aut(A)$ such that $Dim_k(A^{\phi})=3$.
\end{lemma}
\begin{proof}Write $A$ as a first construction $A=J(M_3(k),1)$. Choose $\alpha, \beta,\gamma\in k^*$, each different from $1$, pairwise distinct, such that $\alpha\beta\gamma=1$,  and $a=Diag(\alpha,\beta,\gamma)\neq\pm{1}$. Define $\psi:A\longrightarrow A$ by 
$\psi(x,y,z)=(axa^{-1},ay,za^{-1})$. By the choice of the matrix $a$, it is clear that $\psi$ is semisimple. We have,
$$\psi(x,y,z)=(x,y,z)\Longleftrightarrow (axa^{-1},ay,za^{-1})=(x,y,z)$$
$\Longleftrightarrow$ $x$ centralizes $a$ and $y=z=0$ since $a\neq\pm{1}$ 
Hence $x\in k\times k\times k$ and $A^{\psi}=(k\times k\times k,0,0)$. 
Therefore $Dim_k(A^{\psi})=3$.  
\end{proof}
\noindent
{\bf Proof of Theorem :}Let $\phi\in Aut(A\otimes_k\overline{k})$ be a semisimple automorphism such that $Dim_{\overline{k}}(A\otimes_k\overline{k})=3$, which exists by the lemma. Let $T'\subset G$ be a maximal torus with $\phi\in T'$. For some $g\in G$ we have $gTg^{-1}=T'$. Hence $Dim_k (A^T)\leq 3$. We now apply 
Theorem (\ref{etalefixedpoint}) to a regular element in $T(k)$ to conclude $Dim_k(A^T)\geq 3$. Hence $Dim_k(A^T)=3$ and that $A^T$ is \'{e}tale 
follows from Corollary (\ref{fixedpointcor}).
\begin{corollary} Let $S\subset {\bf Aut}(A)$ be a $k$-torus. Then $Dim(A^S)\geq 3$.
\end{corollary}
\begin{proof} Let $S\subset T$, where $T\subset {\bf Aut}(A)$ is a $k$-maximal torus. It follows that $A^T\subset A^S$ and the result follows from the above theorem.
\end{proof}
\begin{corollary}Let $L\subset A$ be a cubic \'{e}tale subalgebra. Then for $G_L={\bf Aut}(A/L)$, $A^{G_L}=L$. 
\end{corollary}
\begin{proof} We have, for a maximal $k$-torus $T\subset G_L$, $Dim(T)=4$. Hence $T$ is maximal in ${\bf Aut}(A)$ as well and hence we have 
$$L\subset A^{G_L}\subset A^T=L.$$
\end{proof}
\noindent
{\bf Examples :}{\bf (i)} Let $A=J(D,\mu)$ be an Albert algebra with $D$ a degree $3$ division algebra. Let $L\subset D$ be a cubic cyclic subfield. Let 
$p\in L$ with $N_{L/k}(p)=1$. Then the 
automorphism $\mathcal{I}_p:A\rightarrow A,~(x,y,z)\mapsto (x,yp,p^{-1}z)\in{\bf Aut}(A/D_0)$. The subgroup $S=\{\mathcal{I}_p|p\in L,~N_{L/k}(p)=1\}$ is a torus in ${\bf SL}(1,D)$, in fact $S=R^{(1)}_{L/k}(\mathbb{G}_m)$. Hence $S$ is a torus with $Dim(S)=2$. If $(x,y,z)\in A^S$ then we have 
$(x,yp,p^{-1}z)=(x,y,z)$ for all $p\in S$. This implies in particular, $yp=y,~p^{-1}z=z$ for all $p\in S$. Hence $y=z=0$. Therefore $A^S=D_0$ and 
$Dim(A^S)=9$.\\
\noindent
{\bf (ii)} Consider again $A=J(D,\mu)$ with $D$ and $L$ as above. Let $p\in L$ with $N_{L/k}(p)=1$. Consider the automorphism $\psi_{p,1}:A\rightarrow A$ given by $(x,y,z)\mapsto (pxp^{-1},py,zp^{-1})$. Then the subgroup $S=\{\psi_{p,1}|p\in L,~N_{L/k}(p)=1\}$ is a torus in ${\bf Aut}(A)$. Now if $(x,y,z)\in A^S$, then 
$(pxp^{-1},py,zp^{-1})=(x,y,z)$ for all $p\in L,~N_{L/k}(p)=1$. Hence $y=z=0$ and $x\in L$ because $L\subset D$ is a maximal subfield of $D$. Therefore 
$A^S=L$ and $Dim(A^S)=3$. This, together with Example (i) shows that the dimension of the fixed point subalgebra of a torus does not determine the dimension of the torus.  
\section{\bf Tits-Weiss Conjecture}{\label{titsweiss}}
In this section, we take up the proof of the conjecture of Tits and Weiss on the structure group of Albert division algebras over a field. We begin with recalling some classical results on automorphisms of Cayley algebras. We will reprove a theorem of Jacobson (\cite{J2}), which immediately generalizes in the context of Albert algebras. 
\vskip3mm
\noindent
{\bf The case of Cayley algebras :}
\vskip1mm
\noindent
Let $\mathcal{C}$ be a Cayley algebra over a field $k$. Let $a\in\mathcal{C}$. Let $a_L,a_R$ respectively denote the left and the right multiplication maps on $\mathcal{C}$ for $a\in\mathcal{C}$. Then the {\bf flexible law} in $\mathcal{C}$ says that, for $x\in\mathcal{C},~a(xa)=(ax)a$. Hence $a_La_R=a_Ra_L$. Let $U_a=a_La_R$. An automorphism $\eta$ of $\mathcal{C}$ is said to be {\bf inner} if 
$\eta=U_{a_1}\cdots U_{a_r}$ for some $a_1,\cdots a_r\in\mathcal{C}$. We call a nontrivial automorphism $\tau$ of $\mathcal{C}$ a {\bf reflection} if $\tau^2=1$ and if the fixed point subalgebra $\mathcal{B}$ of $\tau$ is $4$--dimensional (hence is a quaternion subalgebra). We can decompose $\mathcal{C}$ as $\mathcal{B}\oplus\mathcal{B}^{\perp}$. Then $\tau$ is identity on $\mathcal{B}$ and $-1$ on $\mathcal{B}^{\perp}$. We have the following two theorems of Jacobson (\cite{J2}), we will give a different proof of the second theorem, only for Cayley division algebras, since that gives us the idea for the Albert algebra case.
\begin{theorem}Every reflection in a Cayley algebra is inner.
\end{theorem}
\begin{theorem}Every automorphism of $\mathcal{C}$ is a product of reflections. In particular, every automorphism of $\mathcal{C}$ is inner.
\end{theorem}
\begin{proof}We will prove the assertion when $\mathcal{C}$ is a division algebra. So let us assume $\mathcal{C}$ is a division algebra and 
$\eta$ is an automorphism of $\mathcal{C}$. Let $\mathcal{C}_0$ denote the subspace of elements of $\mathcal{C}$ of trace $0$. Then $\eta$ maps $\mathcal{C}_0$ to itself. By a theorem of Cartan and Dieudonn\'{e}, there is 
$x_0\neq 0\in\mathcal{C}_0$ such that $\eta(x_0)=x_0$. Let $K=k(x_0)$ be the subalgebra of $\mathcal{C}$ generated by $x_0$. Then $K/k$ is a quadratic field extension. Let $\sigma:K\longrightarrow K$ be the nontrivial $k$-automorphism of $K$. Let $\mathcal{H}\subset\mathcal{C}$ be a quaternion subalgebra containing $K$. By a Skolem-Noether type theorem for composition algebras (see \cite{KMRT}, 33.21), $\sigma$ extends to an automorphism $\mathcal{H}\longrightarrow\mathcal{H}$, which in turn extends to an automorphism $\widetilde{\sigma}:\mathcal{C}\longrightarrow\mathcal{C}$. Hence $\widetilde{\sigma}$ is an automorphism of $\mathcal{C}$ that stabilizes a quaternion subalgebra and hence is a product of reflections by a theorem of M. Wonenburger (\cite{Won}). Consider the automorphism $\psi=\eta\widetilde{\sigma}^{-1}$ of $\mathcal{C}$. Then $\psi$ does not fix $K$ pointwise, since $\widetilde{\sigma}$ does not. Let $L$ be a quadratic extension of $k$ contained in $\mathcal{C}$ that is fixed pointwise by $\psi$. Then $\psi(K)=K$ and $\psi(L)=L$. Hence $\psi$ stabilizes the quaternion subalgebra $\mathcal{B}$ generated by $K$ and $L$. By the aforementioned theorem of M. Wonenburger (\cite{Won}), $\psi$ is a product of reflections. Hence $\eta$ itself is a product of reflections.   
\end{proof}
\noindent
{\bf Back to Albert algebras :}\\
\noindent
We will now adopt a similar approach for Albert division algebras. To begin with, we need to know if every automorphism of an Albert algebra fixes a nonzero element of trace zero. This has already been settled in Theorem (\ref{fixedpoint}). 
We have shown that every automorphism of an Albert division algebra $A$ fixes a cubic subfield 
of $A$. 
\begin{definition} We define an automorphism of an Albert division algebra to be {\bf cyclic} if it fixes a cyclic cubic subfield of $A$. 
\end{definition}
\begin{definition} An automorphism of an Albert algebra is said to be of $A_2$-type if it stabilizes a simple $9$-dimensional subalgebra. 
\end{definition}
This terminology is justified since such an automorphism then belongs to a $k$-subgroup of ${\bf Aut}(A)$ type $A_2$ (see \cite{KMRT}, Chapter IX, 39B). 
\begin{theorem}{\label{cyclic}} Let $A$ be a first Tits construction Albert division algebra over $k$. Let $G=Aut(A)$ and $\phi\in G$. Assume $\phi$ is cyclic. Then $\phi$ is a product of two automorphisms of type $A_2$. 
\end{theorem}
\begin{proof}Let $L\subset A$ be a cubic cyclic subfield that is fixed pointwise by $\phi$. Let $D$ be a degree $3$ central division algebra over $k$ such that $D_+\subset A$ and $L\subset D_+$ (see \cite{PR2}, Cor. 4.6). Then the inclusion $D_+\subset A$ extends to an isomorphism $\theta:A\longrightarrow J(D,\mu)$ for a suitable $\mu\in k^*$ such that $\theta(D_+)=D_0$. Let $Gal(L/k)=<\sigma>$.  By the classical Skolem-Noether theorem, $\sigma:L\longrightarrow L$ extends to an automorphism $\widetilde{\sigma}:D\longrightarrow D$. Hence there exists $a\in D^*$ such that $\widetilde{\sigma}(x)=axa^{-1}$ for all $x\in D$. This in turn extends to an 
automorphism $\widetilde{\sigma}_D:J(D,\mu)\longrightarrow J(D,\mu)$ by $\widetilde{\sigma}_D=\mathcal{I}_a$. 
Consider $\Psi=\phi\theta^{-1}\widetilde{\sigma}_D\theta:A\longrightarrow A$. Note that, for $x\in L,~
\theta^{-1}\widetilde{\sigma}_D\theta(x)=\sigma(x)$. 
Hence $\theta^{-1}\widetilde{\sigma}_D\theta\notin G_L$. Since $\phi\in G_L$, this means that 
$\Psi\notin G_L$. Let $M\subset A$ be a cubic field extension of $k$ such that $\Psi\in G_M$ (this is possible by Corollary \ref{fixedpointcor}). Then 
the subalgebra $S=<L,M>$ is $9$-dimensional (see \cite{PR3}, 2.10) and $\Psi(L)=L,\Psi(M)=M$. 
Hence $\Psi(S)=S$. But the automorphism $\Pi=\theta^{-1}\widetilde{\sigma}_D
\theta$ maps the $9$-dimensional subalgebra $\theta^{-1}(D_0)$ to itself. 
Now $\phi=\Psi\Pi^{-1}$, a product of automorphisms of type $A_2$ as claimed.   
\end{proof}
\begin{theorem}{\label{puretypea2}}Let $A$ be a pure first Tits construction Albert division algebra. Then every automorphism of $A$ is a product of automorphisms of $A$ of type $A_2$.
\end{theorem}
\begin{proof}Let $A$ be a pure first construction Albert division algebra over $k$. Let $\phi\in Aut(A)$. If $\phi$ is cyclic, then by the theorem above, $\phi$ is a product of automorphisms of type $A_2$. If $\phi$ is not cyclic, choose a cyclic cubic subfield $E\subset A$. Let $E=k(x)$ and let $y=\phi(x)$. Then $E=k(x)\subset A$ is a cyclic cubic subfield and $F=\phi(E)$ is also a cyclic cubic subfield of $A$. Since $\phi$ is not cyclic, $x\neq y$. If $E=F$, we must have $\phi(x)=\sigma(x)$, where $Gal(E/k)=<\sigma>$. Let $D$ be a degree $3$ central division algebra over $k$ such that $E\subset D_+\subset A$. Let $\theta:A\simeq J(D,\mu)$ be an isomorphism for some $\mu\in k^*$. By the Skolem-Noether theorem, the automorphism $\sigma:E\rightarrow E$ extends to an automorphism $\widetilde{\sigma}:D\rightarrow D$, which in turn extends to an automorphism $\widetilde{\sigma}_D:J(D,\mu)\rightarrow J(D,\mu)$ as $\widetilde{\sigma}_D=\mathcal{I}_a$, where $a\in D^*$ is such that $\widetilde{\sigma}(z)=aza^{-1},~z\in D^*$. Consider now $\Psi=\theta\widetilde{\sigma}_D^{-1}\theta^{-1}\phi$. Then $\Psi\in Aut(A)$ and 
$$\Psi(x)=\theta\widetilde{\sigma}_D^{-1}\phi(x)=\widetilde{\sigma}^{-1}\sigma(x)=x.$$
Hence $\Psi$ fixes $E$ pointwise and is therefore cyclic. By the theorem above, $\Psi$ is a product of $A_2$-type automorphisms. Also 
$\theta\widetilde{\sigma}_D^{-1}\theta^{-1}$ stabilizes the subalgebra $\theta(D_+)$ and hence is of type $A_2$. Therefore $\phi$ is a product of automorphisms of type $A_2$. If $E\neq F$, we consider $<E,F>$, the subalgebra generated by $E$ and $F$. Then the dimension of this subalgebra is $9$ (see \cite{PR3}, 2.10). Since $A$ is a pure first construction, 
$<E,F>=D_+$ for a central division algebra $D$ of degree $3$ over $k$. The restriction of $\phi$ to $E$ is an isomorphism with $F$. Hence, by the classical Skolem-Noether theorem, there is an automorphism $\theta_D:D_+\longrightarrow D_+$ extending $\phi^{-1}:F\simeq E$. In particular, $\theta_D(y)=\phi^{-1}(y)=x$. 
Let $\widetilde{\theta_D}:A\longrightarrow A$ be an extension of $\theta_D$ to an automorphism of $A$ as before. Consider 
$\phi_1=\widetilde{\theta_D}\phi$. Then we have, 
$$\phi_1(x)=\widetilde{\theta_D}\phi(x)=\widetilde{\theta_D}(\phi(x))=\widetilde{\theta_D}(y)=\theta_D(y)=x.$$ 
Therefore $\phi_1$ fixes $x$ and hence fixes $E$ pointwise, and hence $\phi_1$ is a cyclic automorphism of $A$. By the previous theorem, $\phi_1$ is a product of automorphisms of type $A_2$. But $\phi_1=\widetilde{\theta_D}\phi$ and $\widetilde{\theta_D}$ an 
$A_2$ type automorphism, therefore $\phi$ is a product of automorphisms of type $A_2$. 
\end{proof}
Our next goal is to prove that every automorphism of type $A_2$ for pure first constructions is inner. This will be carried out in two steps. We will first prove that every automorphism which stabilizes a $9$-dimensional subalgebra of $A$ is a product of automorphisms fixing $9$-dimensional subalgebras. We will then prove that an automorphism that fixes a $9$-dimensional subalgebra is indeed inner. While the first assertion is easy to prove, the second one is delicate and needs bit more care. We will prove better results for pure first constructions, while proving general statements for first constructions. We begin with
\begin{theorem}{\label{purefixeda2}} Let $A$ be an Albert division algebra which is a pure first construction. Let $\phi\in Aut(A)$ be such that $\phi$ stabilizes a $9$-dimensional subalgebra. Then $\phi$ is a product of automorphisms that fix $9$-dimensional subalgebras.
\end{theorem}
\begin{proof}Suppose $\phi\in Aut(A)$ stabilizes the subalgebra $D_+\subset A$. The inclusion $D_+\hookrightarrow A$ induces an isomorphism 
$\theta:A\longrightarrow J(D,\mu)$ for some $\mu\in k^*$ such that $\theta(D_+)=D_0$. Then $\psi=\theta\phi\theta^{-1}$ is an automorphism of $J(D,\mu)$ and $\psi$ stabilizes 
$\theta(D_+)=D_0$. By Lemma (\ref{extendaut}), we can find elements $a,b\in D^*$ such that $N_D(a)=N_D(b)$ and $\psi$ is given by 
$$\psi(x,y,z)=(axa^{-1},ayb^{-1},bza^{-1}).$$
Recall that, for $a\in D^*,~\mathcal{I}_a$ denotes the automorphism of $J(D,\mu)$ given by 
$$\mathcal{I}_a(x,y,z)=(axa^{-1},aya^{-1},aza^{-1})$$ 
and, for $p\in SL(1,D)$, $\mathcal{J}_p$ denotes the automorphism of $J(D,\mu)$ given by 
$$\mathcal{J}_p(x,y,z)=(x,yp,p^{-1}z).$$
We have 
\begin{eqnarray*}
\mathcal{J}_{ab^{-1}}\mathcal{I}_a(x,y,z) & = & \mathcal{J}_{ab^{-1}}(axa^{-1},aya^{-1},aza^{-1})\\
& = & (axa^{-1},aya^{-1}ab^{-1},ba^{-1}aza^{-1})\\
& = & (axa^{-1},ayb^{-1},bza^{-1})=\psi(x,y,z).
\end{eqnarray*}
Hence $\psi=\mathcal{J}_{ab^{-1}}\mathcal{I}_a$. We only have to check that $\mathcal{I}_a$ fixes $9$ dimensions. But this is clear, 
since either $a\in k^*$, in which case 
$\mathcal{I}_a=1$ or $L=k(a)$ is a cubic extension of $k$ contained in $D$ and then 
$\mathcal{I}_a$ fixes the subalgebra $S=L\times L\times L$ pointwise. Now 
$$\phi=
\theta^{-1}\psi\theta=(\theta^{-1}\mathcal{J}_{ab^{-1}}\theta)
(\theta^{-1}\mathcal{I}_a\theta).$$
The automorphism $\theta^{-1}\mathcal{J}_{ab^{-1}}\theta$ fixes $\theta^{-1}(D_0)\subset A$ pointwise while the automorphism 
$\theta^{-1}\mathcal{I}_a\theta$ fixes 
the subalgebra $\theta^{-1}(L\times L\times L)\subset A$ pointwise. 
\end{proof}
\vskip2mm
\noindent
{\bf Some computations with $U$ operators (First construction)}
We refer to (\cite{PR3}) for basic formulae on $U$-operators for first construction Albert algebras.
Let $A$ be an Albert algebra over $k$. Assume that $A$ is a first Tits construction, $A=J(D,\mu)$. For $x=(x_0,x_1,x_2),~y=(y_0,y_1,y_2)\in A$, we define 
$$x\times y=(x_0\times y_0-x_1y_2-y_1x_2,\mu^{-1}x_2\times y_2-x_0y_1-y_0x_1,\mu x_1\times y_1-x_2y_0-y_2x_0),$$
$$T(x,y)=t(x_0,y_0)+t(x_1,y_2)+t(x_2,y_1)$$ and 
$$x^{\#}=(x_0^{\#}-x_1x_2,\mu^{-1}x_2^{\#}-x_0x_1,\mu x_1^{\#}-x_2x_0),$$here $t(a,b)=t(ab)$ and $t$ is the reduced trace on $D$. The $U$-operators on $A$ are given by 
$$U_x(y)=T(x,y)x-x^{\#}\times y.$$ 
We have, using these formulae, for $a,b,c\in D$, 
$$(a,0,0)^{\#}=(a^{\#},0,0),~(0,b,0)^{\#}=(0,0,\mu b^{\#}),
~(0,0,c)^{\#}=(0,\mu^{-1}c^{\#},0).$$
We have therefore,
\begin{eqnarray*}
(a,0,0)^{\#}\times(d,e,f) & = & (a^{\#},0,0)\times (d,e,f) ~~~~=  (a^{\#}\times d, -a^{\#}e,-fa^{\#}),\\
(0,b,0)^{\#}\times(d,e,f) & = & (0,0,\mu b^{\#})\times (d,e,f)~~~ = (-\mu eb^{\#},b^{\#}\times f, -\mu b^{\#}d),\\
(0,0,c)^{\#}\times (d,e,f) & = & (0,\mu^{-1}c^{\#},0)\times(d,e,f) = (-\mu^{-1} c^{\#}f,-\mu^{-1}dc^{\#}, c^{\#}\times e).
\end{eqnarray*}
We also note that for $x,y\in D,~t(xy)x-x^{\#}\times y=xyx$.
We have,
$$T((a,0,0),(d,e,f))=t(ad).$$ Hence 
\begin{eqnarray*}
U_{(a,0,0)}(d,e,f) & = & t(ad)(a,0,0)-(a,0,0)^{\#}\times (d,e,f)\\
& = & (t(ad)a,0,0)-(a^{\#}\times d, -a^{\#}e, -fa^{\#})\\
& = & (t(ad)a-a^{\#}\times d, a^{\#}e, fa^{\#})\\
& = & (ada, a^{\#}e, fa^{\#}). 
\end{eqnarray*}
Again,
$$T((0,b,0)(d,e,f))=t(bf)$$ and hence
\begin{eqnarray*}
U_{(0,b,0)}(d,e,f) & = & t(bf)(0,b,0)-(0,b,0)^{\#}\times (d,e,f)\\
& = & (0,t(bf)b,0)-(-\mu eb^{\#},b^{\#}\times f,-\mu b^{\#}d)\\
& = & (\mu eb^{\#}, t(bf)b-b^{\#}\times f,\mu b^{\#}d)\\
& = & (\mu eb^{\#}, bfb, \mu b^{\#}d).
\end{eqnarray*}
Finally, 
$$T((0,0,c),(d,e,f))=t(ce),$$ 
hence 
\begin{eqnarray*}
U_{(0,0,c)}(d,e,f) & = & t(ce)(0,0,c)-(0,0,c)^{\#}\times (d,e,f)\\
& = & (0,0,t(ce)c)-(-\mu^{-1}c^{\#}f,-\mu^{-1}dc^{\#}, c^{\#}\times e)\\
& = & (\mu^{-1}c^{\#}f,\mu^{-1}dc^{\#}, t(ce)c-c^{\#}\times e)\\
& = & (\mu^{-1}c^{\#}f,\mu^{-1}dc^{\#}, cec).
\end{eqnarray*}
We summarize the above computations as 
\begin{eqnarray*}
U_{(a,0,0)}(d,e,f) & = & (ada,a^{\#}e,fa^{\#}),\\
U_{(0,b,0)}(d,e,f) & = & (\mu eb^{\#}, bfb, \mu b^{\#}d),\\
U_{(0,0,c)}(d,e,f) & = & (\mu^{-1}c^{\#}f,\mu^{-1}dc^{\#}, cec).
\end{eqnarray*}
\vskip1mm
\noindent
{\bf Computations with $U$-operators (second construction):}
We now work with $A=J(B,\sigma,u, \mu)$, a second Tits construction Albert algebra. We have, 
for $x=(a,b),~y=(c,d)\in A$
$$x^{\#}=(a^{\#}-bu\sigma(b), \bar{\mu}\sigma(b)^{\#}u^{-1}-ab),~y^{\#}=(c^{\#}-du\sigma(d),\bar{\mu}\sigma(d)^{\#}u^{-1}-cd)$$
$$T(x,y)=t(ac)+t(bu\sigma(d))+\overline{t(bu\sigma(d))}$$
$$x\times y=(a\times c-bu\sigma(d)-du\sigma(b),\bar{\mu}(\sigma(b)\times\sigma(d))u^{-1}-ad-cb)$$
Hence we have $(a,0)^{\#}=(a^{\#},0)$ and 
$$(a,0)^{\#}\times(c,d)=(a^{\#},0)\times (c,d)=(a^{\#}\times c, -a^{\#}d).$$
Now, 
$$U_x(y)=T(x,y)x-x^{\#}\times y.$$
For $x,y\in\mathcal{H}(B,\sigma)$ we note that $U_x(y)=xyx$. 
Therefore we have
\begin{eqnarray*}
U_{(a,0)}(c,d) & = & T((a,0),(c,d))(a,0)-(a,0)^{\#}\times (c,d)\\
               & = & t(ac)(a,0)-(a^{\#}\times c, -a^{\#}d)\\
               & = & (t(ac)a-a^{\#}\times c, a^{\#}d)\\
               & = & (U_a(c),a^{\#}d)\\
               & = & (aca,a^{\#}d).
\end{eqnarray*}
Now we come to the most crucial part of this paper. We will prove some key results, which will facilitate the proof of the Tits-Weiss conjecture. 
First we dispose of a basic computation necessary:
\begin{theorem}{\label{jpinner}}Let $A=J(D,\mu)$ be a first Tits construction Albert algebra and $p\in SL(1,D)$ be a commutator in $D^*$. Then $\mathcal{J}_p$ is inner.
\end{theorem}
\begin{proof}Recall that $\mathcal{J}_p:J(D,\mu)\longrightarrow J(D,\mu)$ is given by $\mathcal{J}_p(x,y,z)=(x,yp,p^{-1})$. Since $p$ is a commutator in $D^*$, we can write $p=jij^{-1}i^{-1}$ for $i,j\in D^*$. Hence $jij^{-1}=pi$, i.e. $ji=pij$ and 
$ji^{-1}j^{-1}=i^{-1}p^{-1}$. Using the formulae on $U$ operators we derived, we have, in $J(D,\mu)$, 
\begin{eqnarray*}
U_{({(ij)}^{-1},0,0)}U_{(i,0,0)}U_{(j,0,0)}U_{(0,1,0)}(x,y,z) & = & U_{({(ij)}^{-1},0,0)}U_{(i,0,0)}U_{(j,0,0)}(\mu y,z,\mu x)\\
& = & U_{({(ij)}^{-1},0,0)}U_{(i,0,0)}(j\mu yj, j^{\#}z,\mu x j^{\#})\\
& = & U_{({(ij)}^{-1},0,0)}(ij\mu yji,i^{\#}j^{\#}z,\mu xj^{\#}i^{\#})
\end{eqnarray*}
\noindent
Now,\\
$U_{({(ij)}^{-1},0,0)}(ij\mu yji,i^{\#}j^{\#}z,\mu xj^{\#}i^{\#})$
\begin{eqnarray*}
& = &({(ij)}^{-1}ij\mu yji(ij)^{-1}, (j^{-1}i^{-1})^{\#}i^{\#}j^{\#}z,\mu xj^{\#}i^{\#}(j^{-1}i^{-1})^{\#})\\
& = & (\mu yjij^{-1}i^{-1},{i^{\#}}^{-1}{j^{\#}}^{-1}i^{\#}j^{\#}z,\mu x)\\
& = & (\mu yp, (jij^{-1}i^{-1})^{\#}z,\mu x)\\
& = & (\mu yp, p^{-1} z,\mu x).
\end{eqnarray*}
We have finally, 
$$U_{(0,0,1)}(\mu yp,p^{-1}z,\mu x)=(\mu^{-1}\mu x, \mu^{-1}\mu yp,p^{-1}z)=(x,yp,p^{-1}z) = \mathcal{J}_p(x,y,z).$$
It follows that $\mathcal{J}_p$ is inner.
\end{proof}
\noindent
{\bf Reduced Whitehead Groups :}Let $D$ be a central simple algebra over a field $k$. Let $SL(1,D)=\{x\in D|Nrd(x)=1\}$ and $D^*=\{x\in D|Nrd(x)\neq 0\}$. Then 
in $D^*$, the commutator subgroup $(D^*,D^*)\subset SL(1,D)$. The {\it reduced Whitehead group} of $D$ is defined by 
$$SK(1,D)=SL(1,D)/(D^*,D^*).$$
We have the following result of Wang (\cite{W}) (see also \cite{T4}, Proposition 2.7):
\begin{theorem}{\label{wang}}
Let $D$ be a central simple algebra over $k$ of square-free index. Then $SK(1,D)=0$. In particular, when degree of $D$ is a prime, $SK(1,D)=0$.
\end{theorem}
\begin{corollary}{\label{sl1dinner}}Let $D$ be a central division algebra of degree $3$ over $k$ and let $p\in SL(1,D)$. Let $J(D,\mu)$ be a first Tits construction Albert algebra over $k$. Then 
$\mathcal{J}_p:J(D,\mu)\longrightarrow J(D,\mu)$ is inner. 
\end{corollary}
\begin{proof} Observe that for $p_1,\cdots,p_r\in SL(1,D),~
\mathcal{J}_{p_1\cdots p_r}=\mathcal{J}_{p_r}\cdots \mathcal{J}_{p_1}$,  and that 
$(D^*,D^*)=SL(1,D)$. Hence the assertion follows from Theorem (\ref{jpinner}).
\end{proof}
\noindent
{\bf Reduced Unitary Whitehead Groups :}Let $(B,\tau)$ be a central simple algebra with a unitary involution $\tau$ over a field $k$. Let $\Sigma_{\tau}(B)$ be the subgroup of $B^*$ generated by $Sym(B,\tau)^*$, where 
$$Sym(B,\tau)=\{x\in B^*|\tau(x)=x\}.$$ 
Let $\Sigma_{\tau}'(B)$ be defined by 
$$\Sigma_{\tau}'(B)=\{x\in B^*|N_B(x)\in k^*\}.$$
Clearly $\Sigma_{\tau}'(B)$ is a subgroup of $B^*$ and $\Sigma_{\tau}(B)\subset\Sigma_{\tau}'(B)$. The {\it reduced unitary Whitehead group} of $B$ is defined by 
$$USK_1(B)=\Sigma_{\tau}'(B)/\Sigma_{\tau}(B).$$
We have the following analogue of Wang's result, due to Yanchevskii (see \cite{KMRT}, Prop. 17.27):
\begin{proposition}{\label{unitarywang}}If $B$ is a central division algebra of prime degree with a unitary involution, then $USK_1(B)=0$.
\end{proposition} 
We now derive an analogue of Corollary(\ref{sl1dinner}) for second Tits constructions.
\begin{corollary}Let $(B,\sigma)$ be a central division algebra of degree $3$ over a quadratic extension $K/k$ with a unitary involution and let $p\in SU(B,\sigma)$. Let $J(B,\sigma,u,\mu)$ be a second Tits construction Albert algebra over $k$. Then $\phi_p:J(B,\sigma,u,\mu)\longrightarrow J(B,\sigma,u,\mu)$ is inner. 
\end{corollary}
\begin{proof}Recall that $\phi_p:J(B,\sigma,u,\mu)\longrightarrow J(B,\sigma,u,\mu)$ is given by $\phi_p(a,b)=(pa\sigma(p),pb)$. Since $p\in SU(B,\sigma)$, we have $p\sigma(p)=1$ and $N_B(p)=pp^{\#}=\sigma(p)^{\#}\sigma(p)=1$. By Proposition \ref{unitarywang} we can write $p=s_1s_2\cdots s_n$ 
for some $s_1,\cdots, s_n\in Sym(B,\sigma)$. Hence $\sigma(p)=s_ns_{n-1}\cdots s_1$ and we have,
\begin{eqnarray*}
\phi_p(a,b)=(pa\sigma(p),pb)& = & (s_1\cdots s_na\sigma(s_1.\cdots s_n),\sigma(p)^{-1}b)\\
                            & = & (s_1\cdots s_nas_n\cdots s_1,\sigma(p)^{\#}b)\\
                            & = & (s_1\cdots s_nas_n\cdots s_1,(s_n\cdots s_1)^{\#}b)\\
                            & = & (s_1s_2\cdots s_nas_ns_{n-1}\cdots s_1, s_1^{\#}s_2^{\#}\cdots s_n^{\#}b)\\
                            & = & U_{(s_1,0)}\cdots U_{(s_n,0)}(a,b),
\end{eqnarray*}
where the last line follows from the computations with $U$-operators in the case of second Tits constructions done above. 
Hence $\phi_p$ is inner. 
\end{proof}
\begin{theorem}{\label{fixedinner}}
Let $A$ be an Albert algebra arising from the first Tits construction. Let $\phi\in Aut(A)$ be an automorphism of $A$ that fixes pointwise a subalgebra of the form 
$D_+$ for a degree $3$ central simple algebra $D$ over $k$. Then $\phi$ is inner. 
\end{theorem}
\begin{proof}The inclusion $D_+\hookrightarrow A$ induces an isomorphism 
$\theta:A\longrightarrow J(D,\mu)$ for some $\mu\in k^*$ such that $\theta(D_+)=D_0$. Let $\psi=\theta\phi\theta^{-1}:J(D,\mu)\longrightarrow J(D,\mu)$. Then $\psi$ fixes $D_0$ pointwise. Hence $\psi$ is given by 
$\psi(x,y,z)=(x,yp,p^{-1}z)$ for some $p\in SL(1,D)$. Now, by the above theorem of Wang, $SK(1,D)=0$ if degree of $D$ is a prime. Recall that the reduced Whitehead group $SK(1,D)=SL(1,D)/(D^*,D^*)$. Hence $p\in SL(1,D)$ is a product of commutators in $D^*$. By the above corollary, $\psi$ is inner. Hence there are elements $a_1,\cdots,a_r\in J(D,\mu)$ such that 
$\psi=U_{a_1}\cdots U_{a_r}$. But then 
$\phi=\theta^{-1} U_{a_1}\cdots U_{a_r}\theta$. Now, for $a\in J(D,\mu),~U_a=
2R_a^2-R_{a^2}$. Hence, for any $x\in A$,
\begin{eqnarray*}
\theta^{-1} U_a\theta(x) & = & \theta^{-1}(2R_a^2-R_{a^2})(\theta(x)) =2\theta^{-1}(a(a\theta(x)))-\theta^{-1}(a^2\theta(x))\\
& = & [2R_{\theta^{-1}(a)}^2-R_{\theta^{-1}(a)^2}](x)=U_{\theta^{-1}(a)}(x).
\end{eqnarray*}
Hence we have, $\phi=U_{\theta^{-1}(a_1)}\cdots U_{\theta^{-1}(a_r)}$.
\end{proof} 
\begin{corollary}{\label{pureinner}} Let $A$ be a pure first construction Albert division algebra over $k$. Then every automorphism of $A$ is inner. In particular, if $A=J(D,\mu)$ be a pure first Tits construction Albert division algebra over $k$ and $a\in D^*$. Then $\mathcal{I}_a:A\longrightarrow A$ defined by 
$\mathcal{I}_a(x,y,z)=(axa^{-1},aya^{-1},aza^{-1})$ is an inner automorphism of $A$.
\end{corollary}
\begin{proof} This follows from Theorem(\ref{puretypea2}), Theorem(\ref{purefixeda2}), and Theorem(\ref{fixedinner}).
\end{proof} 
\noindent
We will now prove that cubes of norm $1$ elements of a degree $3$ central division algebra are product of two commutators. To prove this we need a version of Wedderburn Factorization Theorem for degree $3$ division algebras (see \cite{J3}, Lemma 2.9.8).
This gives an easy proof of the fact that for $A=J(D,\mu)$ and $p\in SL(1,D)$ the automorphism $\mathcal{J}_p^3$ is inner. This of course is a consequence of 
Corollary (\ref{sl1dinner}) proved above, but the proof we present below avoids the use of triviality of $SK(1,D)$. A little terminology first:
\begin{definition}We call an element $a$ of a division algebra $D$ over $k$ {\bf cyclic} if $k(a)$ is a cyclic subfield of $D$. 
\end{definition}
\begin{proposition}{\bf (Wedderburn Factorization Theorem):}Let $D$ be a central division algebra of degree $3$ over $k$ and let $a\in D^*$ be a non-cyclic element. Then the minimum polynomial $f(X)=X^3-\alpha_1 X^2+\alpha_2 X -\alpha_3$ of $a$ over $k$ has a factorization 
$$f(X)=(X-a_2)(X-a_1)(X-a_0)$$
in $D[X]$ with $a_0=a$,
$$c=[a_0a_1]=[a_1a_2]=[a_2a_0]\neq 0,~~ca_ic^{-1}=a_{i+1},$$
where the indices are reduced modulo $3$. Further $c^3=\gamma\in k^*$. 
Here $[xy]=xy-yx,~x,y\in D$.
\end{proposition}
\noindent
We have, 
\begin{proposition}Let $D$ be a degree $3$ central division algebra over $k$ and 
$p\in SL(1,D)$. Then $p^3$ is a product of two commutators in $D^*$.
\end{proposition}
\begin{proof} First let us assume $p\in SL(1,D)$ is cyclic. Then $L=k(p)$ is 
a cyclic cubic subfield of $D$. Hence $L$ is  maximal subfield of $D$ and hence $N_D(p)=N_{L/k}(p)=1=N_{L/k}(p^3)$. Let $Gal(L/k)=<\sigma>$. Then, by Hilbert Theorem-90, 
there exists $q\in L$ such that $p^3=q^{-1}\sigma(q)$. Now, by Skolem-Noether theorem, we can extend $\sigma$ to an automorphism of $D$ and hence, there exists 
$x\in D^*$ such that $\sigma(q)=xqx^{-1}$. Therefore we have 
$$p^3=q^{-1}\sigma(q)=q^{-1}xqx^{-1}\in (D^*,D^*).$$
Now assume $p\in SL(1,D)$ is non-cyclic.  
Let $p=a=p_0$ in the proposition above. Then there are $p_1,p_2, 
c\in D^*$ such that the minimal polynomial $f(X)=X^3-\alpha_1X^2+\alpha_2X-\alpha_3$ of $p_0$ over $k$ factorizes in $D[X]$ as 
$$f(X)=(X-p_2)(X-p_1)(X-p_0)$$ and we have the relations 
$$cp_0c^{-1}=p_1,~cp_1c^{-1}=p_2,~cp_2c^{-1}=p_0.$$ 
It follows that $p_i\in SL(1,D),i=0,1,2$. From these relations, we get
$$cp_0c^{-1}{p_0}^{-1}=p_1{p_0}^{-1},~cp_1c^{-1}{p_1}^{-1}=p_2{p_1}^{-1},~
cp_2c^{-1}{p_2}^{-1}=p_0{p_2}^{-1}.$$Therefore the elements 
$p_1p_o^{-1},~p_2p_1^{-1},~p_0p_2^{-1}\in (D^*,D^*)$ are commutators.
Expanding $f(X)$ and using the fact that $N_D(p)=1$, we get $p_2p_1p_0=1$. Hence we have the relations
$$p_1p_0=p_2^{-1},~p_2p_1=p_0^{-1},~p_0p_2=p_1^{-1}.$$ 
Now $p_0p_1^{-1}=(p_1p_0^{-1})^{-1}\in (D^*,D^*)$ is a commutator, hence substituting $p_1^{-1}=p_0p_2$ we get $p_0(p_0p_2)=p_0^2p_2\in(D^*,D^*)$ is also a commutator. By the conjugacy relations above, we have $cp_2c^{-1}=p_0$, hence $p_2^{-1}cp_2c^{-1}=p_2^{-1}p_0\in (D^*,D^*)$, hence $p_2^{-1}p_0$ is a commutator. 
It follows that $(p_0^2p_2)(p_2^{-1}p_0)=p_0^3\in (D^*,D^*)$ is a product of two commutators. This proves the assertion.
\end{proof}
\begin{corollary}Let $p\in SL(1,D)$. Then $\mathcal{J}_{p^3}=(\mathcal{J}_p)^3$ is inner.
\end{corollary}
\begin{proof}
This is immediate now from Theorem \ref{jpinner}.
\end{proof}
\noindent
Now we prove a general version of the second assertion in Corollary \ref{pureinner} for arbitrary Tits first construction Albert division algebras. 
\begin{theorem}{\label{generalinner}}
Let $A=J(D,\mu)$ be a first Tits construction Albert division algebra. Let $a\in D^*$. 
Then $\mathcal{I}_a:J(D,\mu)\longrightarrow J(D,\mu)$ given by $\mathcal{I}_a(x,y,z)=
(axa^{-1},aya^{-1}aza^{-1})$ is inner. In fact, there are $v_i\in A^*,~1\leq i\leq 4$, satisfying the relation 
$$\mathcal{I}_a=U_{n(a)^{-1}}U_{v_1}\cdots U_{v_4}\in Instr(A).$$
\end{theorem}
\begin{proof}We calculate with $a$ as follows: 
$$
U_{(0,a,0)}U_{(0,0,a)}U_{(a,0,0)}U_{(0,1,0)}(x,y,z)= U_{(0,a,0)}U_{(0,0,a)}U_{(a,0,0)}(\mu y,z,\mu x)$$
$$~~~~~~~~~~~~~~~~~~~~~~~~~~~~= U_{(0,a,0)}U_{(0,0,a)}(\mu aya,n(a)a^{-1}z,\mu n(a)x a^{-1})$$
$$~~~~~~~~~~~~~~~~~~~~~~~~~~~~ = U_{(0,a,0)}(n(a)^2a^{-1}xa^{-1}, n(a)ay,n(a)za)~~~~~$$
$$~~~~~~~~~~~~~~~~~~~~~~~~~~~~~= (\mu n(a)^2aya^{-1},n(a)aza^2,\mu n(a)^3a^{-2}xa^{-1}).~~$$

We now note that $n(a)a^{-3}\in SL(1,D)$. Hence $\mathcal{J}_{n(a)a^{-3}}$ is inner by Corollary (\ref{sl1dinner}). We apply the transformation 
$U_{n(a)^{-1}}
U_{(0,0,1)}\mathcal{J}_{n(a)a^{-3}}$ to the above and get \\
\vskip2mm
$U_{n(a)^{-1}} U_{(0,0,1)}\mathcal{J}_{n(a)a^{-3}}(\mu n(a)^2aya^{-1}, n(a)aza^2,\mu n(a)^3a^{-2}xa^{-1})$ 
\begin{eqnarray*}
& = & n(a)^{-2}U_{(0,0,1)}(\mu n(a)^2aya^{-1},n(a)^2aza^{-1},\mu n(a)^2axa^{-1})\\
& = & n(a)^{-2}(n(a)^2axa^{-1},n(a)^2aya^{-1},n(a)^2aza^{-1})\\
& = & (axa^{-1},aya^{-1},aza^{-1})=(axa^{-1},aya^{-1},aza^{-1})
=\mathcal{I}_a(x,y,z).
\end{eqnarray*}
This proves the assertion.
\end{proof}
\begin{corollary}{\label{generalinnercor}} Let $A$ be a first Tits construction Albert division algebra over $k$ and let $\psi\in Aut(A)$ be such that $\psi(D_+)=D_+$ for a subalgebra $D_+\subset A$. Then $\psi$ is an inner automorphism of $A$.  
. 
\end{corollary}
\begin{proof}As in arguments seen before, we may assume $A=J(D,\mu)$ for a suitable scalar $\mu\in k^*$ and $D_+=D_0$. Hence we have $\psi\in Aut(A)$ such that 
$\psi(D_0)=D_0$. Then it follows that $\psi=\psi_{a,b}$ for some $a,b\in D^*$ with $N_D(a)=N_D(b)$. We have 
$$\psi(x,y,z)=(axa^{-1},ayb^{-1},bza^{-1})=\mathcal{J}_{ab^{-1}}\mathcal{I}_a(x,y,z).$$
Hence $\psi=\mathcal{J}_{ab^{-1}}\mathcal{I}_a$ and we have seen that $\mathcal{I}_a\in Instr(A)$ and $\mathcal{J}_{ab^{-1}}\in Instr(A)$. Hence 
the assertion follows.
\end{proof}
\noindent
\vskip1mm
\noindent
The following lemma and its corollary clarify the action of $Str(A)$ on $A^*$.  
\begin{lemma}{\label{technical}}
Let $A$ be a first Tits construction Albert division algebra. Assume $A=J(D,\mu)$. Then for any $a\in D^*$, there exists $\phi\in Instr(A)$ such that 
$\chi=R_{N_D(a)^{-1}}\phi\in C.Instr(A)$ maps $(1,0,0)$ to $(a,0,0)$. Here $C$ denotes the subgroup of $Str(A)$ consisting of scalar multiplications on $A$. 
\end{lemma}  
\begin{proof}We have for $\chi=\mathcal{R}_{N_D(a)^{-1}}U_{(0,0,1)}U_{(0,a^{\#},0)}$, 
\begin{eqnarray*}
\chi(1,0,0) & = & \mathcal{R}_{N_D(a)^{-1}}U_{(0,0,1)}(0,0,\mu N_D(a)a)\\
& = & \mathcal{R}_{N_D(a)^{-1}}(N_D(a)a,0,0)\\
& = & (a,0,0).
\end{eqnarray*}
Hence $\phi=U_{(0,0,1)}U_{(0,a^{\#},0)}\in Instr(A)$ does the job.
\end{proof}
\noindent
{\bf Remark :} We note that in the above lemma, $\chi\in N_D.Instr(A)$, where, we abuse notation and denote by $N_D$ the subgroup of $C$ consisting of 
scalar multiplications by norms from $D^*$. 
\begin{corollary}{\label{puretechnical}}Let $A$ be a pure Tits fist construction Albert algebra and $a$ an invertible element of $A$. Then there exists $\chi\in 
C.Instr(A)$ such that 
$\chi(1)=a$. 
\end{corollary}
\begin{proof}Let we can imbed $a$ in a subalgebra of $A$ of the form $D_+$ for a degree $3$ central division algebra over $k$. Let $\theta:A\longrightarrow J(D,\mu)$ be an isomorphism extending the inclusion $D_+\hookrightarrow A$ for some $\mu\in k^*$ and $\theta(D_+)=D_0$. Let $\chi'\in C.Instr(J(D,\mu))$ be as in the lemma above. Then $\chi'(1,0,0)=(a,0,0)$. Let $\chi=\theta^{-1}\chi'\theta:A\longrightarrow A$. Then 
$$\chi(1)=\theta^{-1}\chi'\theta(1)=\theta^{-1}\chi'(1,0,0)=\theta^{-1}(a,0,0)
=a.$$
Further, 
$$\chi=\theta^{-1}\chi'\theta=\theta^{-1}\mathcal{R}_{N_D(a)^{-1}}U_{(0,0,1)}U_{(0,a^{\#},0)}\theta=\mathcal{R}_{N_D(a)^{-1}}U_{\theta^{-1}(0,0,1)}
U_{\theta^{-1}(0,a^{\#},0)}.$$ Hence $\chi\in C.Instr(A)$.  
\end{proof}
The theorem below tells us about the norm similarities of an Albert algebra that stabilize a $9$-dimensional subalgebra.  
\begin{theorem}{\label{structuretypea2}}Let $A$ be a first Tits construction Albert algebra. Let $\psi\in Str(A)$ be such that $\psi(D_+)=D_+$ for a subalgebra $D_+\subset A$. Then 
$\psi\in N_DInstr(A)$.
\end{theorem}
\begin{proof}By a simple argument, we may assume that $A=J(D,\mu)$ and identify $D_+$ with $D_0\subset J(D,\mu)$. Let $\psi\in Str(A)$ be such that 
$\psi(D_0)=D_0$. In particular, it follows that $\psi(1,0,0)=(c,0,0)$ for some $c\in D^*$. By the lemma (and its proof) above, 
$$\chi(1,0,0)=\mathcal{R}_{N_D(c)^{-1}}U_{(0,0,1)}U_{(0,c^{\#},0)}(1,0,0)=(c,0,0).$$
Moreover, for $x\in D$, 
\begin{eqnarray*}
\chi(x,0,0) & = & N_D(c)^{-1}U_{(0,0,1)}U_{(0,c^{\#},0)}(x,0,0)\\
& = & N_D(c)^{-1}U_{(0,0,1)}(0,0,\mu c^{\#\#}x)\\
& = & N_D(c)^{-1}(\mu^{-1}\mu N_D(c)cx,0,0)\\
 & = & (cx,0,0).
\end{eqnarray*}
Hence $\chi(D_0)=D_0$. Moreover, we see that 
$$\chi^{-1}\psi(1,0,0)=\chi^{-1}(c,0,0)=(1,0,0).$$
Therefore $\phi=\chi^{-1}\psi\in Aut(A)$ and $\phi(D_0)=D_0$. Hence $\phi=\psi_{a,b}$ for suitable $a,b\in D^*$. We have proved already that 
$\psi_{a,b}\in Instr(A)$ (see Corollary \ref{generalinnercor}). Hence $\psi=\chi\psi_{a,b}\in N_DInstr(A)$, by the lemma above. 
\end{proof} 
\noindent
\noindent
In the spirit of the above theorem, combining Theorem \ref{cyclic} with Lemma \ref{technical} gives,  
\begin{theorem}{\label{cyclicstructure}}Let $A$ be an Albert division algebra over $k$ arising from the first Tits construction. Let $E\subset A$ be a cubic cyclic extension of $k$ contained in $A$ as a subalgebra. Let $\psi\in Str(A)$ be such that $\psi(E)=E$. Then $\psi\in C.Instr(A).H$, where $H$ denotes the subgroup of $Aut(A)$ generated by automorphisms of type $A_2$.
\end{theorem} 
\begin{proof}Since $E/k$ is cyclic cubic extension, there is a subalgebra $D_+\subset A$ such that $E\subset D_+$. We may assume 
(see proof of Theorem \ref{fixedinner}), that $A=J(D,\mu)$ for some $\mu\in k^*$. Hence we may identify $D_+$ with $D_0\subset A$ and $E\subset D_0$. 
Let $\psi(1)=\psi(1,0,0)=(c,0,0),~c\in E$. Let $\chi=\mathcal{R}_{N_E(c)^{-1}}U_{(0,0,1)}U_{(0,c^{\#},0)}$. Then we have, by a computation exactly as in 
the proof of Lemma \ref{technical}, 
$$\chi(1,0,0)=(c,0,0)=\psi(1,0,0).$$
Note also $\chi(E)=E$, in fact, for any $x\in E,~\chi(x,0,0)=(cx,0,0)$, this follows by a computation exactly along the lines of the proof of 
Theorem \ref{structuretypea2}. We have therefore $\chi^{-1}\psi(1,0,0)=(1,0,0)$ and $\chi^{-1}\psi\in Aut(A)$ with $\chi^{-1}\psi(E)=E$. If $\chi^{-1}\psi$ fixes 
$E$ pointwise, then by Theorem \ref{cyclic}, it follows that $\chi^{-1}\psi$ is a product of automorphisms of type $A_2$. Hence $\psi=\chi\phi$, where 
$\chi\in N_EInstr(A)$ and $\phi\in Aut(A)$ is a product of $A_2$ type automorphisms. In the other case, $\chi^{-1}\psi|E=\sigma,~~Gal(E/k)=<\sigma>$. Let $\widetilde{\sigma}$ be an extension of $\sigma$ to an automorphism of $A$ with $\widetilde{\sigma}=\mathcal{I}_a,~a\in D$ (see proof of Theorem \ref{cyclic}). Then 
$\widetilde{\sigma}^{-1}\chi^{-1}\psi\in Aut(A)$ and fixes $E$ pointwise. Hence by the earlier case and the fact that $\widetilde{\sigma}\in N_D.Instr(A)$, it follows that $\psi$ is a product of an element of $C.Instr(A)$ and automorphisms of type $A_2$. 
\end{proof}
\noindent 
In light of these two theorems, one can raise the following question: \\
\noindent
{\bf Question :} Let $A$ be an Albert division algebra over $k$, arising from the first construction. Let $D_+\subset A$ and $X=Str(A)(D_+)$ be the orbit of $D_+$ under the structure group. Does the subgroup $C.Instr(A)$ act transitively on $X$? What if we consider the orbit $Y=Str(A)(E)$ of a cubic cyclic subfield of $A$?\\
\noindent
{\bf Remark :} An affirmative answer to the first would prove the Tits-Weiss conjecture for first constructions. While an affirmative answer to the second would 
prove that for first constructions, any norm similarity is a product of an element of $C.Instr(A)$ and an element of $H$. 
\noindent
For pure first constructions, we have a stronger result:   
\begin{theorem}{\bf (Tits-Weiss Conjecture) :} Let $A$ be a pure first construction Albert division algebra over a field $k$. Then
$$\frac{Str(A)}{C.Instr(A)}=\{1\}.$$
\end{theorem}
\begin{proof}
Let $\phi\in Str(A)$, where $A$ is a pure first Tits construction Albert division algebra over $k$. Let $\phi(1)=a$. Then, by Corollary (\ref{puretechnical}), 
there exists $\chi\in C.Instr(A)$ such that $\chi(a)=1$. Hence $\chi\phi(1)=1$. 
This shows that $\chi\phi$ is a norm isometry. Further, since $\chi\phi(1)=1$, it is an automorphism of $A$. Hence we have reduced the problem to proving that every automorphism of a pure first Tits construction Albert division algebra is inner. This has already been proved! (see Corollary \ref{pureinner}). 
\end{proof} 
\noindent
{\bf The case of Reduced Albert algebras :}One can discuss the Tits-Weiss conjecture in the setting of reduced Albert algebras as well. As we shall see, in this case, it follows rather easily from known results of Jacobson, that the assertion of the conjecture is indeed true. This doesn't seem to be explicit in the literature, we therefore record it below. In the following discussion, we will denote by $Isom(A)$ the subgroup of $Str(A)$ consisting of all isometries of the norm $N$ of $A$. 
\begin{theorem}{\label{reduced}}Let $A$ be a reduced Albert algebra over a field $k$. Then $Str(A)=C.Instr(A)$. 
\end{theorem} 
\begin{proof}Let $\psi\in Str(A)$ and $\psi(1)=a$. Let $\alpha=N(a)$. Then $\alpha\neq 0$ and $U_a$ is invertible. Consider 
$\chi=\mathcal{R}_{\alpha^{-1}}U_a\psi\in C.Instr(A)$. Then we have,
$$\chi(1)=\alpha^{-1}U_a\psi(1)=\alpha^{-1}U_a(a)=\alpha^{-1}a^3.$$
Hence,
$$N(\chi(1))=\alpha^{-3}N(a^3)=\alpha^3N(a)^3=\alpha^{-3}\alpha^3=1.$$
It follows that $\chi\in Isom(A)$. It now suffices to prove that isometries of $N$ are inner. When $A$ is split, this follows from Theorem 9, (\cite{J4}), which proves that the norm preserving group $Isom(A)$ coincides with the subgroup of $Str(A)$ consisting of all mappings $\eta=U_{a_1}\cdots U_{a_r}$ with $\prod N(a_i)=1$. If the coordinate algebra of $A$ is a division algebra, this follows from Theorem 13, (\cite{J1}), which proves that $Isom(A)$ coincides with the subgroup of $Str(A)$ generated by elements $\zeta=\prod U_{a_i}$, where the $a_i$ for a given $\zeta$ are all contained in some nonsplit $9$-dimensional reduced simple subalgebra of degree $3$ and $\prod N(a_i)=1$. 
\end{proof}
\noindent
{\bf $R$-equivalence in $k$-forms of $F_4$:} For the notion of $R$-equivalence, we refer to the book (\cite{Vos}, Chapter 6) or the book 
(\cite{M}, Chapter II, 14). Let $X$ be an irreducible variety over a field $k$, with $X(k)$ nonempty. Manin in \cite{M} introduced the notion of 
$R$-equivalence on $X$; call points $x,y\in X(k)$ $R$-equivalent if there exist points $x=x_0,x_1,\cdots,x_n=y\in X(k)$ and rational maps 
$f_i:\mathbb{P}^1\rightarrow X,~1\leq i\leq n$, all defined over $k$, such that $f_i(0)=x_{i-1},~f_i(\infty)=x_i$. In case $X=G$, a connected algebraic group, the set of elements of $G(k)$ which are $R$-equivalent to $1\in G(k)$ is a normal subgroup $RG(k)$ of $G(k)$. The quotient $G(k)/RG(k)$ is denoted by $G(k)/R$. We now prove 
\begin{theorem}{\label{rtrivial}}Let $A$ be a pure first construction Albert division algebra over a field $k$ and let $G={\bf Aut}(A)$. Then we have 
$$G(k)/R=\{1\}.$$
\end{theorem}
\begin{proof} We have shown in Theorem(\ref{puretypea2}) and Theorem(\ref{purefixeda2}) that every automorphism of a pure first construction Albert division algebra is a product of automorphism fixing $9$-dimensional subalgebras of $A$. Hence it suffices to prove that every automorphism of $A$ that fixes a $9$-dimensional subalgebra, can be connected to the identity automorphism by a rational map over $k$ as above. Let $S\subset A$ be a $9$-dimensional subalgebra. Then, since $A$ is pure first construction, $S=D_+$ for a suitable degree $3$ central division algebra over $k$. Now, the algebraic subgroup of $G={\bf Aut}(A)$ fixing $S$ pointwise is isomorphic to ${\bf SL}(1,D)$ over $k$. We know by Wang's theorem $SK(1,D)=\{1\}$ for degree $3$ central division algebras $D$ over $k$. Also, 
for degree $3$ central division algebras (see \cite{Vos}, Chapter 6),  
$$SK(1,D)\simeq SL(1,D)/R.$$
Hence it follows that $SL(1,D)/R=\{1\}$. From this the result follows.
\end{proof} 
\section{\bf Irregular automorphisms}\label{autos2}
In this section, we prove some results on automorphisms for general Albert division algebras. We will need some notions from the theory of algebraic groups. In this section, we will assume that characteristic of $k$ is not $2$.
\begin{definition}Let $G$ be a connected algebraic group defined over a field $k$ and let $g\in G$. The element $g$ is called {\bf regular} if $Z_G(g)$ the the centralizer of $g$ in $G$ has minimal dimension among all centralizers. If $g\in G$ is not regular, we call $g$ as {\bf irregular} (also called singular sometimes in the literature). 
\end{definition}  
\noindent
Let $A$ be an Albert division algebra arising from Tits first construction. Let $A=J(D,\mu)$ for a degree $3$ central division algebra and 
$\mu\in k^*$. Let $a\in D^*-k^*$ with $Nrd(a)=1$. Let $\psi_{a,1}$ be defined as before,
$$\psi_{a,1}(x,y,z)=(axa^{-1},ay,za^{-1}).$$
We have,
\begin{lemma} With $\phi=\psi_{a,1}$ with $a\neq 1$, we have $A^{\phi}=k(a)$. 
\end{lemma}
\begin{proof} We calculate:
$$\phi(x,y,z)=(x,y,z)\Longleftrightarrow (axa^{-1},ay,za^{-1})=(x,y,z)$$
$$\Longleftrightarrow ax=xa,~ay=y,~z=za.$$
Hence, since $k(a)$ is a maximal commutative subfield of $D$, it follows that, 
$x\in k(a)$. Also, $y=z=0$, since otherwise we would have $a=1$ using the other relations, a contradiction. Therefore $A^{\phi}=k(a)$.  
\end{proof}
\noindent
We will see that if we choose $a\in D^*-k^*$ that is noncyclic, i.e. $k(a)$ is not a Galois extension of $k$, then $\phi$ as above is regular. 
Before we arrive at this, we need a classification of subgroups of $G={\bf Aut}(A)$. 
Let $k$ be a field. Let $\mathcal{G}$ be a simply connected group over $\overline{k}$, an algebraic closure of $k$. We say a connected reductive group $G$ defined over $k$ is of {\bf type} $\mathcal{G}$ if the simply connected cover of the commutator subgroup $(G_{\overline{k}},G_{\overline{k}})$ is a product of groups, each isomorphic to $\mathcal{G}$. We need an 
improved version of a result proved in (\cite{PST}, Proposition 6.1). 
\begin{theorem}{\label{subgroups}}Let $A$ be an Albert division algebra over a field $k$. Let $H\subset G={\bf Aut}(A)$ be a proper connected reductive subgroup defined over $k$ and assume $H$ is not a torus. Then $H$ is of type $A_2$ or $D_4$.  
\end{theorem}
\begin{proof}Let $H$ be a proper connected reductive subgroup of $G$, defined over $k$ and assume $T$ is not a torus. We then have the simply connected cover of $[H,H]$ is isomorphic to $\prod_{i=1}^r R_{L_i/k}H_i$ for some finite extensions $L_i/k$ and absolute almost simple groups $H_i$ defined over $L_i$. Suppose 
some $H_i$ is not of type $A_1,A_2$ or $D_4$. Then, since the (absolute) rank of $G$ (dimension of a maximal torus) is $4$, we must have $[L_i:k]\leq 2$ and $H_i$ becomes isotropic over an extension of degree $2^l$ for some $l$. But then $G$ itself becomes isotropic over such an extension, contradicting the fact that $G$ is anisotropic over $k$ and remains anisotropic over any finite extension of degree coprime to $3$. Hence all $H$ must be of type $A_1,A_2$ or $D_4$. Again, by rank argument, if say $H_1$ is of type $A_1$ then $[L_1:k]=3$ and $r=1$. In this case, 
for a maximal torus $S\subset H_1$ defined over $L_1$, $R_{L_1/k}(S)\subset G$ is a rank $3$ torus in $G$ defined over $k$. By the following proposition, this can not happen. Hence we are done.

\end{proof}  
\begin{proposition}Let $A$ be an Albert division algebra over $k$ and $G={\bf Aut}(A)$. Let $S\subset G$ be a torus defined over $k$. 
Then $S$ has rank $2$ or $4$.
\end{proposition}
\begin{proof}Let $S\subset G$ be a torus defined over $k$. First suppose $S$ is of rank $1$. Then since $G$ is semisimple and anisotropic over $k$, 
$S$ itself must be anisotropic over $k$. Therefore $S=R_{K/k}^{(1)}(\mathbb{G}_m)$ for some quadratic extension $K/k$ (see Example 6, 4.9,\cite{Vos}). 
But then $S$ splits over $K$ and hence $G$ becomes isotropic over $K$, a contradiction to the fact that $G$ remains anisotropic over any extension with degree coprime to $3$, or equivalently, $A$ remains a division algebra over such extensions. Hence $S$ can not have rank $1$. 
Next, suppose $S$ has rank $3$. Let $T$ be a maximal torus defined over $k$ such that $S\subset T\subset G$ over $k$. The quotient $T/S$ is a rank $1$ torus defined over $k$ and we have the exact sequence of tori over $k$,
$$1\longrightarrow S\longrightarrow T\longrightarrow M\longrightarrow 1,$$ 
with $M=T/S$. But then $M$ splits over a quadratic extension $K/k$, hence $M$ has a 
nontrivial character defined over $K$ and since $M=T/S$, this gives a nontrivial character of $T$ defined over $K$. Hence $T$ becomes isotropic over $K$ and hence $G$ becomes isotropic over $K$, a contradiction. Hence rank of $S$ can not be $3$. Therefore $S$ has rank $2$ or $4$.  
\end{proof}
\noindent
{\bf Remark :}Since $G$ has $k$-subgroups of type $A_2$, there are rank $2$ tori in $G$ defined over $k$. For example, if $A=J(D,\mu)$, one has a $k$-rational embedding of ${\bf SL}(1,D)$ in ${\bf Aut}(A)$ and if $A=J(B,\sigma,u,\mu)$ then there is a $k$-embedding of ${\bf SU}(B,\sigma)$ in ${\bf Aut}(A)$ and both these groups are of type $A_2$.  
\begin{corollary}Let $G$ be as in the proposition. Let $H$ be a connected reductive $k$-subgroup of $G$ of maximal absolute rank and not a torus. Then $H$ is of type $A_2$ or $D_4$. 
\end{corollary}
\noindent
We are now in a position to prove 
\begin{theorem}{\label{irreg}}Let $A$ be an Albert division algebra over $k$ and $G={\bf Aut}(A)$. Let $\phi\in G(k)$ be an irregular automorphism of $A$. Then $\phi$ stabilizes a $9$-dimensional subalgebra. 
\end{theorem}
\begin{proof}Let $\phi\in G(k)$ be irregular. Let $L\subset A$ be a cubic subfield fixed pointwise by $\phi$. Then $\phi\in G_L$. Note that $G_L$ is a rank $4$ subgroup of $G$ defined over $k$. Consider the centralizer $Z_G(\phi)$. Since $G$ is simply connected, this is a connected reductive subgroup of $G$ defined over $k$ (see \cite{H}, Theorem 2.11) 
and is not a torus, since $\phi$ is irregular. Hence, by the corollary above, either $\phi$ is central in $G_L$ or $Z_G(\phi)$ is of type $A_2$. 
In the first case, choose a division subalgebra $B\subset A$ of dimension $9$ such that $L\subset B$ and let $1\neq\theta\in G_B(k)\subset G_L(k)$. Then $A^{\theta}=B$ and 
$\phi\theta=\theta\phi$. Hence, for any $x\in B$,
$$\phi(x)=\phi(\theta(x))=\theta(\phi(x)).$$
Hence $\phi(x)\in A^{\theta}=B$. Therefore $\phi(B)=B$.  
In the second case $Z_G(\phi)$ is isomorphic to $A_2\times A_2$. We may assume that $A^{\phi}=L$. Now we note that $A_2\times A_2\nsubseteq D_4$. This follows from the fact that $A_2\times A_2$ has no $8$-dimensional self-dual representation. Hence $Z_G(\phi)(k)\nsubseteq G_L$, using the fact that 
$k$-rational points are dense for a connected reductive subgroup. Let $\theta\in Z_G(\phi)(k)-G_L$. Then 
$A^{\theta}$ is either a cubic subfield $M\subset A$ or a $9$-dimensional subalgebra. In the first case, $\phi$ maps the $9$-dimensional subalgebra generated by $L$ and $M$ to itself. In the second case $\phi$ stabilizes $A^{\theta}$. 
\end{proof}
\noindent
{\bf Example :}Let $A=J(D,\mu)$ be an Albert division algebra and $p\in SL(1,D)$. Consider the automorphism $\mathcal{J}_p:A\longrightarrow A$ given by 
$\mathcal{J}_p(x,y,z)=(x,yp,p^{-1}z)$. Then $A^{\mathcal{J}_p}=D_0$. Clearly, in our earlier notation $H=\{\psi_{x,y}|x\in SL(1,D),~y\in k(p)\}\subset 
Z_G(\mathcal{J}_p)$ and $H\simeq SL(1,D)\times R_{k(p)/k}^{(1)}(\mathbb{G}_m)$ is not a torus. Hence $Z_G(\mathcal{J}_p)$ is not a torus and $\mathcal{J}_p$ is irregular.
The following proposition clarifies things about regular elements in $G(k)$. 
\begin{proposition}{\label{regular}}Let $A$ be an Albert division algebra. Let $\phi\in G(k)$ be such that $A^{\phi}=L$ is a cubic non-cyclic extension of $k$. Then $\phi$ is regular. 
\end{proposition}
\begin{proof}Suppose $A^{\phi}=L$ is a cubic extension of $k$ and $\phi$ is irregular. Then $Z_G(\phi)$ is a connected reductive subgroup which is not a torus. 
Suppose $Z_G(\phi)\subset G_L$. By Theorem \ref{subgroups}, it follows that $Z_G(\phi)$ is a $k$-form of $A_2\times A_2$ or $D_4$. But 
$A_2\times A_2\nsubseteq D_4$, hence $Z_G(\phi)=G_L$ and $\phi $ is central in $G_L$. Recall that $G_L$ is a form of $Spin(8)$. Therefore $\phi$ is a $2$-torsion element, i.e., $\phi^2=1$ and $\phi\neq 1$ since $Dim(A^{\phi})=3$. This implies that there is a nonzero element $v\in A$ such that $\phi(v)=-v$. Since $\phi$ is a norm isometry for the norm on $A$, we have, 
$$N(\phi(v))=N(v)=N(-v)=-N(v),$$
since $N$ is a cubic form and $char(k)\neq 2$. Therefore $N(v)=0$ and hence $A$ is reduced, a contradiction. Hence $Z_G(\phi)\nsubseteq G_L$. Let $\theta\in Z_G(\phi)(k)-G_L$. Then 
$\theta^{-1}\phi\theta=\phi$ and hence $\phi(\theta(x))=\theta(x)$ for all $x\in L$. Since $A^{\phi}=L$, it follows that $\theta(L)=L$. 
Also $\theta\notin G_L$, hence $\theta|_{L}\neq 1$ and $L/k$ must be Galois. Hence if $L/k$ is non-cyclic, $\phi$ must be regular. This proves the proposition.
\end{proof}
\noindent
{\bf Example :} Let $A=J(D,\mu)$ be an Albert division algebra. Let
$a\in SL(1,D)$ be such that $k(a)$ is cyclic. Consider the
automorphism $\phi=\psi_{a,1}:A\longrightarrow A$. Then
$\psi_{a,1}(x,y,z)=(axa^{-1},ay,za^{-1})$ and $A^{\phi}=k(a)$. It is
easily checked that $H=\{\psi_{x,b}|x\in k(a),~b\in
D^*,~Nrd(b)=Nrd(x)=1\}\subset Z_G(\phi)$ and $H\simeq
R_{k(a)/k}^{(1)}(\mathbb{G}_m)\times SL(1,D)$. Hence $H$ is bigger
than a maximal torus and therefore $\phi$ is irregular.  
\vskip3mm
\noindent
{\bf Some general results :} We will now prove some general results for decomposition of automorphisms. We need
\begin{lemma} Let $G$ be a connected algebraic group, defined and anisotropic over a perfect infinite field $k$. Then $G$ is reductive.
\end{lemma}
\begin{proof} Since $G$ is anisotropic over $k$, $G(k)$ has no nontrivial unipotent elements (see \cite{R}). Hence $R_u(G)(k)=\{1\}$. But then the density of $k$-rational points implies that $R_u(G)=\{1\}$. Therefore $G$ is reductive. 
\end{proof}
\begin{proposition}{\label{d1d2}}Let $A$ be an Albert division algebra over a perfect infinite field $k$ and let $G={\bf Aut}(A)$. Let $L$ be a cubic subfield in $A$ and $H=G_L$. Then there exist $9$-dimensional subalgebras $D_1$ and $D_2$ of $A$ such that $D_1\cap D_2=L$ and $H=G_1.G_2\cdots.G_r$, where $G_i={\bf Aut}(A/D_1)$ 
or $G_i={\bf Aut}(A/D_2)$. 
\end{proposition}
\begin{proof}Let $0\neq x_1\in L^{\perp}$ and $D_1$ be the subalgebra generated by $L$ and $x_1$. Let $0\neq x_2\in D_1^{\perp}$ and $D_2$ be the subalgebra generated by $L$ and $x_2$. Then $D_1\cap D_2=L$ and $Dim(D_i)=9$. Moreover, since $D_1\neq D_2$, the subalgebra generated by $D_1$ and $D_2$ equals $A$, since this subalgebra has dimension at least $10$. 
Let $G_i$ be as in the hypothesis. Then $G_i\subset G_L$ and the
subgroup $H$ of $G_L$ generated by $G_1$ and $G_2$ is a closed
connected subgroup defined over $k$. Since $G_L$ is anisotropic, $H$
itself is a $k$-anisotropic nontoral subgroup and by the lemma, $H$ is reductive. Hence $H$ is connected reductive non-toral and contains $G_i$
properly. Therefore, by Theorem (\ref{subgroups}), $H$ must be of type $A_2\times A_2$ or
$D_4$. But $A_2\times A_2\nsubseteq D_4$, hence $H$ must be of type
$D_4$ and therefore $H=G_L$. That $H=G_1\cdots G_r$ with $G_i={\bf Aut}(A/D_1)$ or
${\bf Aut}(A/D_2)$ follows from a standard theorem in algebraic group theory (see for example \cite{Spr}, Corollary 2.2.7).
\end{proof}
We can define the length of and element of $G_L$ as the length of the smallest expression of the element as a product of elements from $G_i$. We then have,
\begin{corollary}Let $\phi\in G_L(k)$ be of length at most $2$. Then $\phi=\phi_1\phi_2$ with $\phi_i\in G_i(k)$ with $i=1$ or $2$.
\end{corollary}
\begin{proof}Let $\phi=\in G_L(k)$ with length at most $2$. If length of $\phi$ is $1$, then clearly $\phi$ belongs to $G_i$ for $i=1$ or $2$. Now suppose 
$\phi=\phi_1\phi_2$. Let $\Gamma=Gal(\overline{k}/k)$. Then for any $\sigma\in \Gamma$ we have,
$$\phi=\sigma(\phi)=\phi_1\phi_2\sigma(\phi_1)\sigma(\phi_2).$$
Hence 
$$\phi_1^{-1}\sigma(\phi_1)=\phi_2\sigma(\phi_2)^{-1}\in G_1\cap G_2.$$
But $G_1\cap G_2=\{1\}$, since any automorphism in the intersection must fix both $D_1$ and $D_2$ pointwise and hence must be identity, as $D_1$ and $D_2$ generate $A$. Therefore $\sigma(\phi_i)=\phi_i,~i=1,2$ and hence $\phi_i$ belongs to $G_1$ or $G_2$ as asserted. 
\end{proof}
\section{Concluding Remarks}
This paper raises several questions. For example, it seems likely that every automorphism of an Albert division algebra over $k$ is a 
product of automorphisms of type $A_2$. We have been able to establish this only for irregular automorphisms in general. This result might give a hold on the general case of the Tits-Weiss conjecture, if one can prove that $A_2$-type automorphisms are inner modulo scalar multiplications. 
The study of norm similarities seems possible via the study of automorphisms in general. Skolem-Noether type results might come in handy (as they did in the pure case) in tackling general first constructions. However, there is always an obstruction for Skolem-Noether type extension theorem for isomorphisms of cubic subfields (see \cite{PT}). One could study the $k$-subgroups of ${\bf Str}(A)$ for an Albert algebra and study branching rules for 
the representation on $A$ for these subgroups. This may help in proving the factorization of a norm similarity into simpler ones. Studying fixed points of subgroups and classifying subgroups (defined over $k$) via their fixed point subalgebras is another tool that one could exploit. This is tricky though, 
as we have seen, this approach fails to yield much information on tori. However, one maybe able to do better for semisimple $k$-subgroups of ${\bf Aut}(A)$. Let $T\subset{\bf Aut}(A)$ be a $k$-maximal torus. 
We have shown that $T$ fixes a cubic \'{e}tale subalgebra $L$ of $A$. Hence we get an embedding (over $L$) of $T$ in ${\bf SO}(Q_L)$, where $Q_L$ denotes the Springer form of $L$ (see \cite{PR2}). It maybe of interest to study the decomposition of the $8$-dimensional $L$-representation of $T$ thus obtained and find invariant subspaces over $k$ for various maximal $k$-tori. This may require studying the twisted octonion corresponding to the Springer decomposition of $A$ with respect to $L$ and studying the corresponding automorphism group, which is a twisted $k$-form of $D_4$. Again, by methods of Section 6, it maybe possible to prove $G(k)/R=1$ for an {\it arbitrary} Albert division algebra $A$, $G={\bf Aut}(A)$. Hence it maybe a worthwhile effort to prove generalizations of results of Section 6, at the level of $k$ points.    
\vskip5mm
\noindent
{\bf Acknowledgements}\\
\noindent
I would like to express my deepest gratitude to H. P. Petersson for meticulously reading the preliminary version of this paper and pointing out several errors in that. I thank Richard Weiss, Tom De Medts, B. Sury and Dipendra Prasad for generously sharing their knowledge with me on the subject. 
Katrin Tent and Linus Kramer supported my visit to University of Wuerzburg back in 2004, where this work was conceived, their help is gratefully acknowledged. 
The author thanks the Abdus Salam I.C.T.P., Trieste, for its hospitality on various occasions.
I thank Amit Kulshrestha, Anupam Singh and Shripad Garge for their interest in the work. I would like to thank G. Prasad for supplying the reference \cite{T4} for the result of Wang.  

\vskip5mm
%Indian Statistical Institute, 7-S.J.S. Sansanwal Marg, New Delhi 110016, India\\
%\center email: maneesh.thakur@gmail.com

\end{document}